\documentclass[preprint,11pt]{imsart}

\RequirePackage[OT1]{fontenc}
\RequirePackage{amsthm,amsmath}
\RequirePackage[numbers]{natbib}
\RequirePackage[colorlinks,citecolor=blue,urlcolor=blue]{hyperref}

\usepackage{amssymb}

\startlocaldefs
\numberwithin{equation}{section}
\theoremstyle{plain}
\newtheorem{theorem}{Theorem}[section]

\newtheorem{lemma}{Lemma}[section]
\newtheorem{corollary}{Corollary}[section]

\theoremstyle{definition}
\newtheorem{remark}[theorem]{Remark}
\endlocaldefs

\newcommand{\bea}{\begin{eqnarray}}
\newcommand{\ena}{\end{eqnarray}}
\newcommand{\beq}{\begin{equation}}
\newcommand{\enq}{\end{equation}}
\newcommand{\beas}{\begin{eqnarray*}}
\newcommand{\enas}{\end{eqnarray*}}

\newcommand{\EE}{{\mathbb{E}} }

\newcommand{\bs}{{\bf{s}}}
\newcommand{\bS}{{\bf{S}}}
\newcommand{\AAA}{{\mathcal{A}}}

\begin{document}

\begin{frontmatter}
\title{Chi-square approximation by Stein's method with application to Pearson's statistic}
\runtitle{Chi-square approximation by Stein's method}

\begin{aug}
\author{\fnms{Robert E.} \snm{Gaunt}}
\,
\author{\fnms{Alastair} \snm{Pickett}}
\and
\author{\fnms{Gesine} \snm{Reinert}}

\runauthor{Gaunt et al.}

\affiliation{University of Oxford\thanksmark{m1}}

\address{Department of Statistics\\
University of Oxford\\
24-29 St$.$ Giles'\\
Oxford OX1 3LB\\
United Kingdom }
\end{aug}

\begin{abstract}
 This paper concerns the development of Stein's method for chi-square approximation and its application to problems in statistics.  New bounds for the derivatives of the solution of the gamma Stein equation are obtained.  These bounds involve both the shape parameter and the order of the derivative.  Subsequently Stein's method for chi-square approximation is applied to bound the distributional distance between Pearson's statistic and its limiting chi-square distribution, measured using smooth test functions.  In combination with the use of symmetry arguments, Stein' method yields explicit bounds on this distributional distance of order $n^{-1}$.
\end{abstract}



\end{frontmatter}

\section{Introduction}
One of the first statistical tests which a student learns is based on Pearson's chi-square statistic, denoted by $X^2$ and commonly used to test goodness-of-fit in classification problems. Under the null hypothesis of fit to (any) specified distribution over $m$ classes, $X^2$ converges in distribution to a chi-square random variable with $m-1$ degrees of freedom, as the sample size increases. There are many rules-of-thumb as to when it is valid to use Pearson's test, generally based on experience and simulation experiments; the most famous such rule, oft quoted in school textbooks, is that the expected values in each class should be  at least 5 under the null hypothesis. This restriction is now considered by many authors to be conservative, although for some datasets, any such stipulation quickly becomes a severe handicap, see for example \cite{roscoebyers}. Error bounds can be used to rigorously improve such rules-of-thumb and base them on the experimenter's particular needs. 

The most common method of analysing the asymptotics of $X^2$ is by way of normal approximation, using the Lindeberg-Levy multivariate central limit theorem. Knowledge of the convergence rate of the central limit theorem shows that $X^2$ converges to its limit at rate of at least  $n^{-1/2}$, where $n$ is the overall size of the sample. This strategy has been applied to the approximation of quadratic forms to obtain bounds on the distance to a normal distribution using Stein's method, see for example \cite{de15}.

In this paper, we consider convergence in the square directly, using Stein's method, to obtain an explicit bound on the rate of convergence of order $n^{-1}$ for smooth test functions. This may seem counter-intuitive at first sight, but can be attributed to the fact that there is an additional symmetry structure to the problem that is not exploited when using  a normal approximation; it is this symmetry which gives rise to the improvement in the order of the chi-square approximation. 
Another motivation for studying chi-square approximations directly is that sometimes the underlying normality does not hold; see for example \cite{fanetal}.  

In order to obtain explicit bounds for such a chi-square approximation we employ Stein's method. Stein's method was  introduced in  \cite{stein} for assessing the distance between a probability distribution and the normal distribution.  At the heart of Stein's method for normal approximation is an inhomogeneous differential equation, known as the Stein equation:
\begin{equation} \label{normal equation} f'(x)-xf(x)=h(x)-\Phi h,
\end{equation} 
where $\Phi h$ denotes the quantity $\mathbb{E}h(Z)$ for $Z\sim N(0,1)$.  Evaluating both sides of (\ref{normal equation}) at a random variable $W$ and taking expectations gives
\begin{equation} \label{expect} \mathbb{E}[f'(W)-Wf(W)]=\mathbb{E}h(W)-\Phi h.
\end{equation}
Thus, the quantity $\mathbb{E}h(W)-\Phi h$ can be bounded by solving the Stein equation (\ref{normal equation}) and then bounding the left-hand side of (\ref{expect}). Associated with \eqref{normal equation} is the Stein operator
\begin{equation} \label{normalop}
{\mathcal{A}}f(x) =  f'(x)-xf(x),
\end{equation} 
defined for differentiable functions $f$.

Over the years, Stein's method has been extended to many other distributions, such as the Poisson \cite{chen 0}, multinomial \cite{loh}, exponential \cite{chatterjee}, \cite{pekoz1}, Laplace \cite{pike}, variance-gamma \cite{gaunt}, and the gamma distribution \cite{luk}, \cite{nourdin1}, which we develop further in this paper.  For multivariate normal approximations, the method was first adapted in  \cite{barbour2} and  \cite{gotze}, viewing the normal distribution as the stationary distribution of an Ornstein-Uhlenbeck diffusion, and using the generator of this diffusion as a characterising operator \eqref{normalop} for the normal distribution.

This generator approach was used by  \cite{luk} to extend Stein's method to the gamma distribution.   Through this method \cite{luk} obtained the following Stein equation for the the $\Gamma(r,\lambda)$ distribution with probability density function $\frac{\lambda^r}{\Gamma(r)}x^{r-1}\mathrm{e}^{-\lambda x}$, $x>0$:
\begin{equation} \label{delight} xf''(x)+(r-\lambda x)f'(x)=h(x)-\Gamma_{r,\lambda} h, 
\end{equation}
where $\Gamma_{r,\lambda} h$ denotes the quantity $\mathbb{E}h(X)$ for $X\sim \Gamma(r,\lambda)$ (this characterisation of the gamma distribution was known from  \cite{diaconis}).  It therefore follows that a Stein operator 
 for the chi-square distribution with $p$ degrees of freedom is 
\begin{equation}\label{chiop} {\mathcal{A}}_p f(x) = xf''(x)+\frac{1}{2}(p- x)f'(x),
\end{equation}
defined for all twice differentiable functions $f$, 
and a 
Stein equation for the chi-square distribution with $p$ degrees of freedom is 
\begin{equation}\label{chieqn}xf''(x)+\frac{1}{2}(p- x)f'(x)=h(x)-\chi_{(p)}^2h,
\end{equation}
where $\chi_{(p)}^2 h$ denotes the quantity $\mathbb{E}h(X)$ for $X\sim\chi_{(p)}^2$.  

\cite{luk}  obtained the second essential ingredient of Stein's method for gamma approximation by bounding the derivatives of the solution $f$ of the gamma Stein equation (\ref{delight}).  Let $h^{(k)}$ denote the $k$-th derivative of a ($k$ times differentiable) function $h$, and let 
\begin{align*}
\mathcal{C}_{\lambda,k}&= \{h:\mathbb{R}^+\rightarrow\mathbb{R}:  \exists c>0, a < \lambda \mbox{ such that } \forall  x \in \mathbb{R}, \; \ell = 0, 1, \ldots, k-1,\\
&\quad\quad | h^{(\ell)}(x) | \le c\mathrm{e}^{ax} \mbox{ and } h^{(k-1)} \mbox{ is absolutely continuous} \} .
\end{align*}
 Then, for $h\in \mathcal{C}_{\lambda,k}$,
\begin{equation}\label{lukluk2}\|f^{(k)}\|\leq\frac{\|h^{(k)}\|}{k\lambda}, \qquad k\geq 1,
\end{equation}
where $\|f\|:=\|f\|_\infty=\sup_{x>0}|f(x)|$.

An alternative bound was obtained by  \cite{gaunt thesis}, improving a bound of  \cite{pickett}.  If $h\in \mathcal{C}_{\lambda,k-1}$ then
\begin{equation}\label{gamma22}\|f^{(k)}\|\leq \bigg\{\frac{\sqrt{2\pi}+\mathrm{e}^{-1}}{\sqrt{r+k-1}}+\frac{2}{r+k-1}\bigg\} \|h^{(k-1)}\|,  \qquad k\geq 1,
\end{equation}
where $h^{(0)}\equiv h$.  (Indeed \cite{luk}, \cite{gaunt thesis} and \cite{pickett} imposed stronger conditions on the test function $h$, although by slightly modifying their proofs one can weaken the assumptions to those stated above.)  The bound \eqref{gamma22} involves one fewer derivative of the test function $h$ than \eqref{lukluk2}, and also - and more importantly in the context of chi-square approximation - \eqref{gamma22} involves the shape parameter $r$.  For the $\chi_{(p)}^2$ distribution, inequality (\ref{lukluk2}) gives a bound $\|f^{(k)}\|\leq\frac{2}{k}\|h^{(k)}\|$, whereas bound (\ref{gamma22}) is of order $p^{-1/2}$ for large $p$.  

In Section 2, we obtain new bounds for the solution of the gamma Stein equation.  Of particular interest is the bound of Theorem \ref{gamthm33}, which is of order $p^{-1}$ for large $p$: 
\begin{equation}\label{useful1}\|f^{(k)}\|\leq\frac{2}{r+k-1}\big(3\|h^{(k-1)}\|+2\lambda\|h^{(k-2)}\|\big), \qquad k\geq 2.
\end{equation}
The $p^{-1}$ rate, which is optimal (see Remark \ref{optimalpr}), means that (\ref{useful1}) is more suited to the applications  considered in this paper than either (\ref{lukluk2}) or (\ref{gamma22}).

The rest of this paper applies of Stein's method for chi-square approximation to problems in statistics.  Before bounding the distance between Pearson's statistic and its limiting chi-square distribution, we illustrate Stein's method for chi-square approximation by considering a simpler example in which the random variables are independent and identically distributed.  Specifically, let $\mathbf{X}$ be a $n \times d$ matrix of i.i.d$.$ random variables $X_{ij}$ with zero mean and unit variance.  Then the statistic
\begin{equation}
\label{cbsj}W_d = \frac{1}{n} \sum_{j=1}^d \bigg( \sum_{i=1}^n X_{ij} \bigg)^2
\end{equation}
is asymptotically $\chi^2_{(d)}$ distributed, by the central limit theorem.  From the Berry-Ess\'{e}en theorem, one might
expect $W_d$ to converge (in the weak convergence sense) at a rate of order $n^{-1/2}$.  In contrast, in Theorem \ref{rightresult}, through the use of symmetry arguments, we are able to obtain bounds of order $n^{-1}$ for smooth test functions.  Our bound (\ref{useful1}) for the solution of the Stein equation allows an improvement on the $O(dn^{-1})$ and $O(d^{1/2}n^{-1})$ bounds that would result from an application of (\ref{lukluk2}) or (\ref{gamma22}).  Though, for non-smooth test functions, we expect a $n^{-1/2}$ rate to be optimal; see Remark \ref{Binomial_remark}.      

In Section 4, we use Stein's method for chi-square approximation to obtain bounds for the distance between Pearson's statistic and its limiting distribution.  This is a more challenging application, because, unlike (\ref{cbsj}), Pearson's statistic cannot be written in terms of i.i.d$.$ random variables (knowledge of whether the outcome of the $j$-th trial falls in the first $m-1$ cells allows one to determine whether it fell in the $m$-th cell).  
In particular, by using symmetry considerations, we are able to obtain a bound of order $n^{-1}$ for smooth test functions (Theorem \ref{Pearson_thm_non_int}).  This is the first $O(n^{-1})$ bound for the rate of convergence of Pearson's statistic that holds for all $m\geq2$.  While  \cite{gu03} proved that the rate of convergence of Pearson's statistic in Kolmogorov distance is $O(n^{-1})$ for all $m\geq6$, they did not give an explicit bound.  In Theorem \ref{pearsqrtn}, we obtain a $O(n^{-1/2})$ bound that holds for smooth test functions which has smaller constants and a better dependence on $m$ than the bound of Theorem \ref{Pearson_thm_non_int}.  Both of the  bounds depend on $n$ and the null hypothesis cell classification probabilities $p_1,\ldots,p_m$ in the correct manner, in that they tend to zero if and only if $np_*\rightarrow\infty$, where $p_*=\min_{1\leq i\leq m}p_i$.  A simple consequence of these bounds is a Kolmogorov distance bound for the rate of convergence of Pearson's statistic (Corollary \ref{kolcol}).  The dependence on $n$ is suboptimal, but Corollary \ref{kolcol} is the first Kolmogorov distance bound for Pearson's statistic that tends to zero if and only if $np_*\rightarrow\infty$. 

The rest of the article is organised as follows.  Section 2 gives new bounds for the derivatives of the solution of the gamma Stein equation.  In Section 3, we demonstrate Stein's method for chi-square approximation by considering an example in which the random variables are independent and identically distributed.  In particular, symmetry considerations can be used to obtain bounds of order $n^{-1}$.  In Section 4, we obtain bounds for the distance between Pearson's statistic and its limiting chi-square distribution, one of which is of order $n^{-1}$ for smooth test functions.  A bound for the rate of convergence of Pearson's statistic in Kolmogorov distance appears as a corollary.  Proofs of technical results are postponed to  Section \ref{appendix}.

\section{Stein's method for the gamma distribution}  

First we  briefly review some of the existing literature of Stein's method for the gamma distribution.  We shall need some of this theory to obtain our  bounds for the derivatives of the solution of the gamma Stein equation.

The following characterisation of the gamma distribution (see  \cite{diaconis} and  \cite{luk}) is the starting point for Stein's method for gamma approximation.  The random variable $X$ has the $\Gamma(\lambda,r)$ distribution if and only if
\begin{equation} \label{delight333} \mathbb{E}[Xf''(X)+(r-\lambda X)f'(X)]=0 
\end{equation}
for all twice differentiable functions $f:\mathbb{R}^+\rightarrow\mathbb{R}$ which are such that the expectations $\mathbb{E}|Zf''(Z)|$, $\mathbb{E}|f'(Z)|$ and $\mathbb{E}|Zf'(Z)|$ are finite for $Z\sim \Gamma(r,\lambda)$.  This characterisation leads to the $\Gamma(r,\lambda)$ Stein equation (\ref{delight}).

It is straightforward to verify that
\begin{align}\label{soln10}f'(x)&=\frac{1}{xp(x)}\int_0^{x}(h(t)-\Gamma_{r,\lambda}h)p(t)\,\mathrm{d}t\\
\label{soln20}&=-\frac{1}{xp(x)}\int_x^{\infty}(h(t)-\Gamma_{r,\lambda}h)p(t)\,\mathrm{d}t,
\end{align}
where $p(x)=\frac{\lambda^r}{\Gamma(r)}x^{r-1}\mathrm{e}^{-\lambda x}$, solves the $\Gamma(r,\lambda)$ Stein equation (see  \cite{stein2}, p$.$ 59, Lemma 1). 
The representations (\ref{soln10}) and (\ref{soln20}) of the solution become difficult to work with if one is interested in bounding higher order derivatives.  \cite{luk} used probabilistic arguments to obtain an alternative representation of the solution for which it was possible to write down a simple formula for derivatives of general order.  From this formula for the $k$-th order derivative \cite{luk} deduced the bound (\ref{lukluk2}), and  \cite{gaunt thesis} used this formula and some more involved calculations to obtain (\ref{gamma22}) (see also  \cite{pickett}).

With the introductory results now stated, we turn our attention to deriving our order $r^{-1}$ bound for the solution of $\Gamma(r,\lambda)$ Stein equation.  In proving our result we do not need to make use of either (\ref{soln10}) or  the formula in \cite{luk} for the solution; the bound follows from a simple application of  (\ref{lukluk2}) and the following lemma.

\begin{lemma}\label{xfr0}Let $f'$ be the solution (\ref{soln10}) of the Stein equation (\ref{delight}).  Then, if $h:\mathbb{R}^+\rightarrow\mathbb{R}$ is bounded,
\begin{equation}\label{xfr1}\|xf''(x)\|\leq 2 \|h-\Gamma_{r,\lambda}h\|\leq 4\|h\|.
\end{equation}
Suppose now that $h\in \mathcal{C}_{\lambda,k}$ for $k\geq 1$.  Then
\begin{equation}\label{xfr2}\|xf^{(k+2)}(x)\|\leq 4\|h^{(k)}\|
\end{equation}
and
\begin{equation}\label{xfr3}\|xf^{(k+1)}(x)\|\leq\frac{4}{\lambda}\big(2+\sqrt{r+k}\big)\|h^{(k)}\|.
\end{equation}
\end{lemma}

\begin{proof}
First we show that, for bounded $h$,
\begin{equation}\label{chdbch}\|xf''(x)\|\leq 2\|h-\Gamma_{r,\lambda}h\|.
\end{equation}
From \eqref{delight}  and the triangle inequality, 
$$ |x f''(x) | \le  \| h-\Gamma_{r,\lambda}h  \|  + |  (r - \lambda x) f'(x)|. $$ 
Now, for $0 < x < \frac{\lambda}{r}$, from \eqref{soln10}, 
\begin{align*}
| (r - \lambda x) f'(x)| &=  \frac{ r - \lambda x }{xp(x)}\int_0^{x}(h(t)-\Gamma_{r,\lambda}h)p(t)\,\mathrm{d}t\\
&\le    \| h-\Gamma_{r,\lambda}h  \|   \frac{ r - \lambda x }{xp(x)}\int_0^{x} p(t)\,\mathrm{d}t.
\end{align*}
Note that with $p(x)=\frac{\lambda^r}{\Gamma(r)}x^{r-1}\mathrm{e}^{-\lambda x}$ it holds for all $t>0$ that 
$ (t p(t))' = (r - \lambda t) p(t).$
Hence, for $0 < x < \frac{\lambda}{r}$, 
\begin{align*}
 0 < \frac{ r - \lambda x }{xp(x)}\int_0^{x} p(t)\,\mathrm{d}t &\le 
  \frac{ r - \lambda x }{xp(x)}\int_0^{x}  \frac{ r - \lambda t}{r - \lambda x} p(t)\,\mathrm{d}t \\
  &=   \frac{ 1 }{xp(x)}\int_0^{x}   (t p(t))' \,\mathrm{d}t =1.
\end{align*}
With an analogous argument for the case $x > \frac{\lambda}{r}$, using representation \eqref{soln20}, the bound \eqref{chdbch} follows, and with it 
inequality (\ref{xfr1}).

\medskip We now use (\ref{xfr1}) to prove inequality (\ref{xfr2}).  The technique used to achieve this is similar to the proof of Proposition 4.2 of  \cite{dobler beta} (the technique has since been further developed in \cite{dgv15}).  The $\Gamma(r,\lambda)$ Stein equation with a straightforward induction on $k$ gives 
\begin{equation}\label{kderv}xf^{(k+2)}(x)+(r+k-\lambda x)f^{(k+1)}(x)-k\lambda f^{(k)}(x)=h^{(k)}(x).
\end{equation}
Note that, since $h\in\mathcal{C}_{\lambda,k}$, it follows from (\ref{lukluk2}) and (\ref{gamma22}) that $f^{(k)}(x)$ and $f^{(k-1)}(x)$ exist and are bounded, and consequently $xf^{(k+2)}(x)$ also exists and is bounded.  Rearranging (\ref{kderv}), 
\begin{equation}\label{kderv11}xf^{(k+2)}(x)+(r+k-\lambda x)f^{(k+1)}(x)=h^{(k)}(x)+k\lambda f^{(k)}(x),
\end{equation}
which we recognise as the $\Gamma(r+k,\lambda)$ Stein equation with $f$ replaced by $f^{(k)}$, and $h(x)-\Gamma_{r,\lambda}h$ replaced by $\tilde{h}(x):=h^{(k)}(x)+k\lambda f^{(k)}(x)$.  Now $\tilde{h}(x)$ is bounded and has zero mean with respect to the $\Gamma(r+k,\lambda)$ distribution.  Indeed, for $X\sim\Gamma(r+k,\lambda)$,
\begin{align*}\mathbb{E}\tilde{h}(X)=\mathbb{E}[Xf^{(k+2)}(X)+(r+k-\lambda X)f^{(k+1)}(X)]=0,
\end{align*}
by characterisation (\ref{delight333}) of the $\Gamma(r+k,\lambda)$ distribution, since the expectations $\mathbb{E}|Xf^{(k+2)}(X)|$, $\mathbb{E}|f^{(k+1)}(X)|$ and $\mathbb{E}|Xf^{(k+1)}(X)|$ are finite.  Hence, it follows from (\ref{xfr1}) that
\begin{equation*}
\|xf^{(k+2)}(x)\|\leq 2\|h^{(k)}(x)+k\lambda f^{(k)}(x)\|\leq 2(\|h^{(k)}\|+k\lambda \|f^{(k)}\|)\leq 4\|h^{(k)}\|,
\end{equation*}
where the final inequality follows from (\ref{lukluk2}).

\medskip Finally, we prove inequality (\ref{xfr2}).  From (\ref{kderv}), 
\begin{equation*}\lambda xf^{(k+1)}(x)=-h^{(k)}(x)-k\lambda f^{(k)}(x)+(r+k)f^{(k+1)}(x)+xf^{(k+2)}(x),
\end{equation*}
and applying the triangle inequality gives
\begin{equation*}\lambda \|xf^{(k+1)}\|\leq\|h^{(k)}\|+k\lambda \|f^{(k)}\|+(r+k)\|f^{(k+1)}\|+\|xf^{(k+2)}(x)\|.
\end{equation*}
Inequality (\ref{xfr3}) now follows on bounding $\|f^{(k)}\|$, $\|f^{(k+1)}\|$ and $\|xf^{(k+2)}(x)\|$ using inequalities (\ref{lukluk2}), (\ref{gamma22}) and (\ref{xfr2}), respectively.
\end{proof}

The following theorem  follows easily from Lemma \ref{xfr0}.

\begin{theorem}\label{gamthm33}Let $k\geq2$ and suppose $h\in \mathcal{C}_{\lambda,k-1}$ and that $h^{(k-2)}$ is bounded.  Then
\begin{equation}\label{rorder}\|f^{(k)}\|\leq\frac{2}{r+k-1}\big(3\|h^{(k-1)}\|+2\lambda\|h^{(k-2)}\|\big),
\end{equation}
where $h^{(0)}\equiv h$.  In particular, 
 for the derivatives of the solution of the $\chi_{(p)}^2$ Stein equation (\ref{chieqn}), we have, for $k \ge 2$,
\begin{equation}\label{chibound} \|f^{(k)}\|\leq\frac{4}{p+2}\big(3\|h^{(k-1)}\|+\|h^{(k-2)}\|\big).
\end{equation} 
\end{theorem}

\begin{proof}From (\ref{kderv}) it follows that
\begin{equation*}f^{(k)}(x)=\frac{1}{r+k-1}\big\{h^{(k-1)}(x)+\lambda(k-1)f^{(k-1)}(x)-xf^{(k+1)}(x)+\lambda xf^{(k)}(x)\big\},
\end{equation*}
and applying the triangle inequality gives
\begin{equation*}\|f^{(k)}\|\leq\frac{1}{r+k-1}\big\{\|h^{(k-1)}\|+\lambda(k-1)\|f^{(k-1)}\|+\|xf^{(k+1)}(x)\|+\lambda \|xf^{(k)}(x)\|\big\}.
\end{equation*}
Using (\ref{lukluk2}) to bound $\|f^{(k-1)}\|$ and using (\ref{xfr2}) to bound $\|xf^{(k)}(x)\|$ and $\|xf^{(k+1)}(x)\|$ gives the desired bound.  In the case $k=2$, we use (\ref{xfr1}) to bound $\|xf^{(k)}(x)\|$. For the last assertion take $r=\frac12 p $ and $\lambda = \frac12$, yielding
$$\|f^{(k)}\|\leq\frac{4}{p+2k-2}\big(3\|h^{(k-1)}\|+\|h^{(k-2)}\|\big);  $$
 \eqref{chibound} follows by noting that $k-1 \ge 1$.  
\end{proof}

\begin{remark}\label{optimalpr}The bound (\ref{rorder}) of Theorem \ref{gamthm33} is of order $r^{-1}$ as $r\rightarrow\infty$.  This is indeed the optimal order, which can be seen as follows. Evaluating both sides of the $\Gamma(r,\lambda)$ Stein equation at $x=0$ gives 
\[f'(0)=\frac{1}{r}[h(0)-\Gamma_{r,\lambda}h].\]
Also, evaluating both sides of equation (\ref{kderv11}) at $x=0$ gives that
\[f^{(k)}(0)=\frac{1}{r+k-1}\{h^{(k-1)}(0)+(k-1)\lambda f^{(k-1)}(0)\}, \qquad k\geq 2.\]
We therefore have that
\[f''(0)=\frac{1}{r+1}\bigg(h'(0)+\frac{\lambda}{r} [h(0)-\Gamma_{r,\lambda}h]\bigg),\]
which for a general test function $h$ is of order $r^{-1}$.  Repeating this procedure shows that the optimal order for $\|f^{(k)}\|$ is $O(r^{-1})$ as $r\rightarrow\infty$.
\end{remark}

\section{A squared central limit theorem}

let $\mathbf{X}$ be a $n \times d$ matrix of i.i.d$.$ random variables $X_{ij}$ with zero mean and unit variance.  In this section, we use symmetry considerations to obtain a bound of order $n^{-1}$ for the distance between the statistic
\begin{equation}
W_d = \frac{1}{n} \sum_{j=1}^d \bigg( \sum_{i=1}^n X_{ij} \bigg)^2
\label{W_d_eqn}
\end{equation}
and its limiting $\chi_{(d)}^2$ distribution.  To elucidate the proof, we firstly consider the case $d=1$; the general $d$ case follows easily as $W_d$ is a linear sum of $W_1$.  For ease of notation, set $W \equiv W_1$ and $X_i \equiv X_{i1}$.  Let  $C_b^k(\mathbb{R}^+)$ denote the class of bounded functions $h:\mathbb{R}^+\rightarrow\mathbb{R}$ for which $h^{(k)}$ exists and derivatives up to $k$-th order are bounded.  Stein's method for the chi-square distribution yields the following result:

\begin{theorem}\label{chiiid}

Let  $X,X_1, \ldots, X_n$ be i.i.d$.$ random variables with $\mathbb{E}X=0$, $\mathbb{E}X^2=1$ and $\mathbb{E}X^8<\infty$,
and let $W \equiv W_1$ be defined as per equation \eqref{W_d_eqn}.  Then, for $h \in C_b^3(\mathbb{R}^+)$, 
\begin{equation} |\mathbb{E} h(W) - \chi^2_{(1)}h| \leq \frac{4\mathbb{E}X^8}{3n} \{ \alpha_0 \| h\| +  \alpha_1 \|h'\| + \alpha_2 \|h''\| + \alpha_3 \|h^{(3)}\| \}, \label{Katie} \end{equation}
where $\chi^2_{(1)}h$ denotes the expectation of $h(T)$ for $T \sim \chi^2_{(1)}$ and
\begin{eqnarray*}
\alpha_0 &=& 2 + 69 |\mathbb{E}X^3|   , \\
\alpha_1 &=&  38 + 654 |\mathbb{E}X^3|,   \\
\alpha_2 &=&   203+ 1781 |\mathbb{E}X^3|  , \\
\alpha_3 &=& 321+ 1320 |\mathbb{E}X^3|    .
\end{eqnarray*} 
\end{theorem}

\begin{proof}
The proof comes in two parts. The first part includes expansions and bounds, whereas the second part includes the symmetry argument. 

\medskip
\noindent{\bf{Proof Part I: Expansions and bounding}}

\medskip 
 Let $S = \frac{1}{\sqrt{n}} \sum_{1}^n X_j$, so that $W = S^2$.  We also let $S^{(i)}=S-\frac{1}{\sqrt{n}}X_i$ and note that $S^{(i)}$ and $X_i$ are independent.  Using the Stein equation for a $\chi^2_{(1)}$ random variable, we require a bound on the expression $\mathbb{E}[ W f''(W) + \frac{1}{2} (1 - W)f'(W)]$.  Define the function $g\!: \mathbb{R} \rightarrow \mathbb{R}$ by $g(s) = \frac{1}{4}f(s^2).$  Then simple differentiation shows that
\begin{equation}
 \mathbb{E}[ W f''(W) + \tfrac{1}{2}  (1 - W)f'(W)]  = \mathbb{E} [g''(S) -  S g'(S)].
\label{Mercury}
\end{equation}
The right-hand side of this equality can be recognised as the quantity to be bounded in the Stein equation for the standard normal distribution, with $g'$ instead of $g$ (see equation \eqref{normal equation}). Thus, our problem is reduced from that of a $\chi^2_{(1)}$ to that of a normal.

As the $X_i$ are identically distributed,
\[\mathbb{E} S g'(S)=\frac{1}{\sqrt{n}}\sum_{i=1}^n\mathbb{E}X_ig'(S)
=\sqrt{n}\mathbb{E}X_1g'(S).\]
Taylor expansion of $g''(S)$ and $g'(S)$ about $S^{(1)}$, using  independence and that $\mathbb{E}X_1=0$ and $\mathbb{E}X_1^2=1$ gives 
\begin{align*}\mathbb{E} [g''(S) -  S g'(S)]&=\mathbb{E}g''(S^{(1)})+\frac{1}{\sqrt{n}}\mathbb{E}X_1\mathbb{E}g^{(3)}(S^{(1)})+R_1 \\
&\quad-\sqrt{n}\mathbb{E}X_1\mathbb{E}g'(S^{(1)})-\mathbb{E}X_1^2\mathbb{E}g''(S^{(1)})-R_2-R_3\\
&=R_1- N -R_2,
\end{align*}
where
\begin{eqnarray*}R_1&=&\frac{1}{2n }\mathbb{E}X_1^2g^{(4)} \bigg(S^{(1)} + \theta_{1} \frac{X_1}{\sqrt{n}}\bigg), \\
N&=&\frac{\mathbb{E}X^3}{2\sqrt{n}} \mathbb{E} g^{(3)}(S^{(1)}), \mbox{ and } \\
R_2&=&\frac{1}{6n}\mathbb{E} X_1^4 g^{(4)}\bigg(S^{(1)} + \theta_{2}\frac{X_1}{\sqrt{n}}\bigg)
\end{eqnarray*}
for some $\theta_1, \theta_2 \in (0,1)$. 
In order to bound these terms  note that 
\begin{equation*}g^{(4)}(s) =  3 f''(s^2)+ 12 s^2 f^{(3)}(s^2)+4 s^4 f^{(4)}(s^2) 
\end{equation*}
 so that 
\bea \label{gbound}
|  g^{(4)}(s) |  \le  3 \| f''\| + 12 s^2 \|f^{(3)}\| + 4 s^4 \| f^{(4)}\|. 
\ena
  Let $\xi_{\theta}= S^{(1)} + \theta \frac{X_1}{\sqrt{n}} .$ 
In bounding $R_1$, and throughout this proof, we shall use that, for $\theta\in(0,1)$, 
\begin{equation} \label{oneofthem} 
\mathbb{E}|X_1^p\xi_{\theta}^q|=\mathbb\mathbb{E}\bigg|X_1^p\bigg(S^{(1)}+\theta\frac{X_1}{\sqrt{n}}\bigg)^q\bigg|\leq 2^{q-1}\bigg[\mathbb{E}|X|^p\mathbb{E}|S^{(1)}|^q+\frac{\mathbb{E}|X|^{p+q}}{n^{q/2}}\bigg],
\end{equation}
as $|a+b|^q\leq 2^{q-1}(|a|^q+|b|^q)$ for $q\geq 1$.

We begin by bounding $R_1$;
\begin{align*} 
|R_1| &= \frac{1}{2n}\mathbb{E}|X_1^2\{3f''(\xi_{\theta_1}^2)+12\xi_{\theta_1}^2f^{(3)}(\xi_{\theta}^2)+4\xi_{\theta_1}^4f^{(4)}(\xi_{\theta_1}^2)\}| \\
&\leq \frac{3}{2n}\|f''\|+\frac{12}{n}\|f^{(3)}\|\bigg[\mathbb{E}X^2\mathbb{E}(S^{(1)})^2+\frac{\mathbb{E}X^4}{n}\bigg]\\
&\quad+\frac{16}{n}\|f^{(4)}\|\bigg[\mathbb{E}X^2\mathbb{E}(S^{(1)})^4+\frac{\mathbb{E}X^6}{n^2}\bigg] \\
&\leq \frac{3}{2n}\|f''\|+\frac{12}{n}\|f^{(3)}\|\bigg[1+\frac{\mathbb{E}X^4}{n}\bigg]+\frac{16}{n}\|f^{(4)}\|\bigg[3+\frac{\mathbb{E}X^4}{n}+\frac{\mathbb{E}X^6}{n^2}\bigg],
\end{align*}
using that $\mathbb{E} (S^{(1)})^2 =\frac{n-1}{n} <1$ and the inequality
\[\mathbb{E} (S^{(1)})^4 = \frac{(n-1)\mathbb{E}X^4+3(n-1)(n-2)}{n^2}<3+\frac{\mathbb{E}X^4}{n}\]
to obtain the last inequality.  The bounding of $R_2$ is similar, with the modification that  the order of the moments of $X$ is increased by 2:
\begin{align*}|R_2|&\leq\frac{1}{2n}\mathbb{E}X^4\|f''\|+\frac{4}{n}\|f^{(3)}\|\bigg[\mathbb{E}X^4+\frac{\mathbb{E}X^6}{n}\bigg]\\
&\quad+\frac{16}{3n}\|f^{(4)}\|\bigg[3\mathbb{E}X^4+\frac{(\mathbb{E}X^4)^2}{n}+\frac{\mathbb{E}X^8}{n^2}\bigg].
\end{align*}

It remains to bound $N$.  
By Taylor expanding $g^{(3)}(S^{(1)})$ about $S$, 
\[N=\frac{\mathbb{E}X^3}{2\sqrt{n}} \mathbb{E} g^{(3)}(S)+R_3,\]
where
\[R_3=-\frac{\mathbb{E}X_1^3}{2 n }\mathbb{E}X_1g^{(4)}\bigg(S^{(1)}+\theta_3\frac{X_1}{\sqrt{n}}\bigg)\]
for some $\theta_3\in (0,1)$.  We bound $R_3$ similarly to $R_1$ and $R_2$,
\[|R_3|\leq|\mathbb{E}X^3|\bigg\{\!\frac{3}{2n}\|f''\|+\frac{12}{n}\|f^{(3)}\|\bigg[1+\frac{\mathbb{E}|X|^3}{n}\bigg]+\frac{16}{n}\|f^{(4)}\|\bigg[3+\frac{\mathbb{E}X^4}{n}+\frac{\mathbb{E}|X|^5}{n^2}\bigg]\!\bigg\}.\]

It remains to show that $\frac{\mathbb{E}X^3}{2\sqrt{n}} \mathbb{E} g^{(3)}(S)$ is of order $n^{-1/2}$. Here the symmetry argument enters.

\noindent{\bf{Proof Part II: Symmetry argument for optimal rate}}

\medskip 
We shall approach this problem using a form of normal approximation: since  $S \approx N(0,1)$, the $O(n^{-1/2})$ central limit theorem convergence rate to the standard normal Stein equation can be applied with test functions  $g^{(3)}$. 
Using \eqref{normal equation}  the Stein equation with test function $g^{(3)}$ for a standard normal random variable is given
by
\begin{equation} \psi'(x) - x \psi(x) = g^{(3)}(x) - \Phi g^{(3)} . \label{Stein-normal} \end{equation}
Now the  test function  $g^{(3)}(s) = 3s f''(s^2) + 2 s^3 f^{(3)}(s^2)$ is an odd function ($g^{(3)}(-x) = - g^{(3)}(x)$), and so must then be the density of $g^{(3)}(Z)$ for $Z$ standard normal, and hence must have zero mean, giving that $\Phi g^{(3)} =0$. The following lemma gives bounds on the solution of this Stein equation; its proof is in Section \ref{appendix}. 

\begin{lemma}\label{normaltech} 
For the solution $\psi$ of \eqref{Stein-normal} we have that
\begin{align}
| \psi (x) | & \le  3 \|f'' \| +  2 (x^2+2)  \| f^{(3)}\|  , \label{bound1} \\
| x \psi '(x) | &\le 
6 x^2  \|f'' \| +  4 x^2 (x^2+1)  \| f^{(3)}\|   
 , \label{bound2} \\
 \label{bound3}
| \psi '' (x) | &\le 6(2 x^2 +1) \|f'' \| +  2(2 x^4+ 3 x^2 + 8) \| f^{(3)}\|   +  4 x^4 \| f^{(4)}\|. 
\end{align} 
\end{lemma}

Now, performing Taylor expansions as
in the first part of the proof, 
\begin{equation*} \EE g^{(3)}(S) = \mathbb{E}[\psi'(S)- S \psi(S)] = R_{4} + R_{5}, \end{equation*}
where
\beas
R_{4} &=&\frac{1}{\sqrt{n}}  \mathbb{E} X_1 \psi''\left( S^{(1)} + \theta_{4} \frac{X_1}{\sqrt{n}} \right) , \\
R_{5} &=& -\frac{1}{2 \sqrt{n}}  \mathbb{E} X_1^3 \psi''\left( S^{(1)} + \theta_{5} \frac{X_1}{\sqrt{n}} \right)
\enas
for some $\theta_{4}, \theta_{5} \in (0,1)$.  
Using \eqref{bound3} and $\EE |X| \le 1$ we obtain that 
\begin{align*}
| R_4| &\le \frac{1}{\sqrt{n}}  \mathbb{E} |X_1  \psi''( \xi_{\theta_{4}} ) | \\
&\leq \frac{1}{\sqrt{n}}   \Big\{ 6 \| f''\| \mathbb{E} |X_1  ( 2 \xi_{\theta_{4}}^2 + 1)|  + 2  \| f^{(3)}\| \mathbb{E} |X_1  ( 2 \xi_{\theta_{4}}^4 + 3 \xi_{\theta_{4}}^2  + 8)|  \\
&\quad+ 4  \| f^{(4)}\| \mathbb{E} |X_1  \xi_{\theta_{4}}^4| \Big\} \\
&\le   \frac{1}{\sqrt{n}}   \Big\{ 6 \| f''\|+ 16   \| f^{(3)}\|    + \mathbb{E} |X_1   \xi_{\theta_{4}}^2| ( 12 \| f''\|  +6  \| f^{(3)}\|) \\
&\quad+ 4 \mathbb{E} |X_1  \xi_{\theta_{4}}^4 | ( \| f^{(3)}\|  +   \| f^{(4)}\|)  \Big\} .
\end{align*}
The last expression we bound with \eqref{oneofthem} to obtain 
\begin{align*}
| R_4| &\le  \frac{1}{\sqrt{n}}   \bigg\{ 6 \| f''\|+ 16   \| f^{(3)}\|    + \left( 1 + \frac{\EE |X|^3}{n} \right)  ( 12 \| f''\|  +6  \| f^{(3)}\|)    \\
&\quad+ 32  \bigg( 3 + \frac{\EE X^4}{n} + \frac{\EE |X|^5}{n^2} \bigg) ( \| f^{(3)}\|  +   \| f^{(4)}\|) 
 \bigg\} 
\\
&=  \frac{1}{\sqrt{n}}   \bigg\{  6\bigg[ 3 +    \frac{2 \EE |X|^3}{n}  \bigg] \| f''\|+    
2 \bigg[ 59 + \frac{3\EE |X|^3}{n} + \frac{16\EE X^4}{n}  \\
&\quad+ \frac{16\EE |X|^5}{n^2}  \bigg]  \| f^{(3)}\|        + 32  \bigg[ 3 + \frac{\EE X^4}{n} + \frac{\EE |X|^5}{n^2}  \bigg]  \| f^{(4)}\| 
 \bigg\}  
. 
\end{align*} 
Similarly, 
\begin{align*}
| R_5| &\le  \frac{1}{2\sqrt{n}}     \Big\{ 6 \| f''\| \mathbb{E} |X_1^3 ( 2 \xi_{\theta_{4}}^2 + 1)|  + 2  \| f^{(3)}\| \mathbb{E} |X_1^3 ( 2 \xi_{\theta_{4}}^4 + 3 \xi_{\theta_{4}}^2  + 8)|  \\
&\quad+ 4  \| f^{(4)}\| \mathbb{E} |X_1 ^3 \xi_{\theta_{4}}^4| \Big\}\\
&\le  \frac{1}{\sqrt{n}}   \bigg\{  3\bigg[ 3\mathbb{E}|X|^3 +    \frac{2 \EE |X|^5}{n}  \bigg] \| f''\|\!+\!    \bigg[ 59\mathbb{E}|X|^3 + \frac{3\EE |X|^5}{n} + \frac{16\mathbb{E}|X|^3\EE X^4}{n}  \\
&\quad+ \frac{16\EE |X|^7}{n^2}  \bigg]  \| f^{(3)}\|        + 16 \bigg[ 3\mathbb{E}|X|^3 + \frac{\mathbb{E}|X|^3\EE X^4}{n} + \frac{\EE |X|^7}{n^2}  \bigg]  \| f^{(4)}\| 
 \bigg\}    
. 
\end{align*} 
To conclude, we have shown that
\begin{equation*} |\mathbb{E} h(W) - \chi^2_{(1)}h| \leq |R_1|+|R_2|+|R_3|+\frac{3|\mathbb{E}X^3|}{2\sqrt{n}}(|R_{4}|+|R_{5}|),
 \end{equation*}
and summing up these remainder terms gives
\begin{equation*} 
 |\mathbb{E} h(W) - \chi^2_{(1)}h| \leq
\frac{1}{n} \{ \beta_1 \|f''\| + \beta_2 \|f^{(3)}\| + \beta_3 \|f^{(4)}\| \},  
\end{equation*}
where 
\begin{align*}
\beta_1 &=  
\frac{3}{2}+\frac{\mathbb{E}X^4}{2}+|\mathbb{E}X^3|\bigg[\frac{57}{2}+\bigg(\frac{27}{2}+\frac{18}{n}\bigg)\mathbb{E}|X|^3+\frac{9\mathbb{E}|X|^5}{n}\bigg] 
,\\
\beta_2 &=
12+4\bigg(1+\frac{3}{n}\bigg)\mathbb{E}X^4+\frac{4\mathbb{E}X^6}{n}+|\mathbb{E}X^3|\bigg[189+\bigg(\frac{177}{2}+\frac{21}{n}\bigg)\mathbb{E}|X|^3\\
&\quad+\frac{48\mathbb{E}X^4}{n}+\frac{24\mathbb{E}|X|^3\mathbb{E}X^4}{n}+\frac{1}{n}\bigg(\frac{9}{2}+\frac{48}{n}\bigg)\mathbb{E}|X|^5+\frac{24\mathbb{E}|X|^7}{n^2}\bigg] 
 , \\
\beta_3 &=
48+16\bigg(1+\frac{1}{n}\bigg)\mathbb{E}X^4+\frac{16(\mathbb{E}X^4)^2}{3n}+\frac{16\mathbb{E}X^6}{n^2}+\frac{16\mathbb{E}X^8}{3n^2}+|\mathbb{E}X^3|\bigg[192 \\
&\quad+72\mathbb{E}|X|^3+\frac{64\mathbb{E}X^4}{n}+\frac{24\mathbb{E}|X|^3\mathbb{E}X^4}{n}+\frac{64\mathbb{E}|X|^5}{n^2}+\frac{24\mathbb{E}|X|^7}{n^2}\bigg] 
.
\end{align*}
The $\beta_i$ can be simplified using  that $n\geq 1$ and $1\leq\mathbb{E}|X|^a\leq\mathbb{E}|X|^b$ for $2\leq a\leq b$, as well as the inequalities $\mathbb{E}|X|^3(\mathbb{E}X^4)^{3/4}\leq \mathbb{E}|X|^3\mathbb{E}X^4\leq\mathbb{E}|X|^7$ and $(\mathbb{E}X^4)^2\leq\mathbb{E}X^8$, which follow from H\"{o}lder's inequality.  Carrying out this simplification gives that
\beas
\beta_1 &\le &  (2 + 69 |\mathbb{E}X^3|)  \mathbb{E}X^8 
,\\
\beta_2 &\le & ( 32 + 447 |\mathbb{E}X^3| ) \mathbb{E}X^8 
 , \\
\beta_3 &\leq& ( 107 + 
440 |\mathbb{E}X^3| ) \mathbb{E}X^8 
.
\enas 
Finally, using inequality (\ref{chibound}) to translate bounds on the derivatives of the solution $f$ to bounds on the derivatives of the test function $h$ completes the proof of Theorem \ref{chiiid}.  
\end{proof} 

\begin{remark}
Instead of using \eqref{rorder} to bound the derivatives of $f$ we could have used \eqref{xfr2}, which would have reduced the moment requirements on $X$ to sixth order and also weakened the conditions on the test function $h$, and would have resulted in the bound 
\beas 
 |\mathbb{E} h(W) - \chi^2_{(1)}h| \leq \frac{\mathbb{E}X^6}{n} \{\gamma_0\|h\|+ \gamma_1 \|h'\| + \gamma_2 \| h'' \| \} , 
\enas 
where 
\begin{equation*}\gamma_i=A_i+B_i|\mathbb{E}X^3|, \qquad i=0,1,2,
\end{equation*} 
for universal constants $A_i$ and $B_i$.  This approach does not easily adapt to give the right order in $d$ for $W_d$, whereas the previous approach does, as will be seen in Theorem \ref{rightresult}. 
\end{remark} 

Moving onto the case of $d > 1$,  the following generalisation of Theorem \ref{chiiid} is almost immediate:

\begin{theorem} \label{rightresult} 
Suppose the $X_{ij}$ are defined as before with bounded eighth moment, and let $W_d$ be defined as in equation (\ref{W_d_eqn}).
Then, for a function $h\in C_b^3(\mathbb{R}^+)$ and for any positive integer $d$, 
\begin{equation} | \mathbb{E} h(W_d) - \chi^2_{(d)}h| \leq \frac{4d\mathbb{E}X^8}{(d+2)n} \{\alpha_0\|h\|+ \alpha_1 \|h'\| + \alpha_2 \|h''\| + \alpha_3 \|h^{(3)}\| \},
\label{thm_result}
\end{equation}
where the $\alpha_i$ are as in the statement of Theorem \ref{chiiid}.
 \end{theorem}

\begin{proof}Define $W_{(j)}=\frac{1}{n}\left(\sum_{i=1}^{n}X_{ij}\right)^2$, so that $W_d=\sum_{j=1}^dW_{(j)}$.  Using the $\chi^2_{(d)}$ Stein equation and conditioning gives
\begin{align*}\mathbb{E}h(W_d)-\chi^2_{(d)}h&=\mathbb{E}[W_df''(W_d)+\tfrac{1}{2}(d-W_d)f'(W_d)]  \\
&=\sum_{j=1}^d\mathbb{E}[W_{(j)}f''(W_d)+\tfrac{1}{2}(1-W_{(j)})f'(W_d)]  \\
&=\sum_{j=1}^d\mathbb{E}[\mathbb{E}[W_{(j)}f''(W_d)+\tfrac{1}{2}(1-W_{(j)})f'(W_d) \: | \\
&\quad\: W_{(1)},\ldots,W_{(j-1)},W_{(j+1)},\ldots,W_{(d)})]]. 
\end{align*}
Since $\|g^{(n)}(x+c)\|=\|g^{(n)}(x)\|$ for any constant $c$,  bound (\ref{Katie}) from Theorem \ref{chiiid} can be used to bound the above expression, which yields (\ref{thm_result}).
\end{proof}

\begin{remark}\label{Binomial_remark}
The premise that the test function must be smooth is vital.  Consider the following example in the univariate case with the single point test function $h \equiv \chi_{\{0\}}$.  Let $X_i$, $i = 1, \ldots, n = 2k$, be random variables taking values in the set $\{ -1,1 \}$ with equal probability.  Then $\mathbb{E} X_i = 0$, $\mathrm{Var} X_i = 1$ and
\begin{equation*}
\mathbb{E} h(W) = \mathbb{P}\Big( \sum\nolimits_i X_i = 0\Big) = \binom{2k}{k} \bigg( \frac{1}{2} \bigg)^{2k} \approx \frac{1}{\sqrt{\pi k}} = \sqrt{\frac{2}{\pi n}},
\end{equation*}
by Stirling's approximation.  Furthermore, $\chi^2_{(1)} h = \mathbb{P} (\chi^2_{(1)} = 0) = 0$, and hence the total variation distance between the distribution of $W$ and the $\chi^2_{(1)}$ distribution is of order $n^{-1/2}$. It is an open question whether a bound of order $n^{-1}$ in Wasserstein distance can be achieved.
\end{remark}

\begin{remark}The bound (\ref{thm_result})  allows for $d\rightarrow\infty$.  Of course for large $d$, the statistic $W_d$ will be approximately normally distributed because the chi-square distribution approaches the normal distribution as $d$ increases.
\end{remark}

\section{Application to Pearson's statistic}

Now we tackle the Pearson chi-square goodness-of-fit test, introduced in 1900 in  \cite{pearson}.

\begin{theorem}[Pearson's $\chi^2$ test] Consider $n$ independent trials, with each trial leading to a unique classification over $m$ classes.  Let the vector $\mathbf{p} = (p_1,\ldots,p_m)$ represent the non-zero classification probabilities, and let $(U_1,\ldots,U_m)$ represent the observed numbers arising in each class.  Then Pearson's chi-square statistic, given by
\begin{equation} W = \sum_{j=1}^m \frac{(U_j - n p_j)^2}{n p_j}, \end{equation}
is asymptotically $\chi_{(m-1)}^2$ distributed.
\end{theorem}

The aim of this section is to ascertain how fast $W$ converges (in the sense of weak convergence) in terms of $n$, $m$ and $p_1,\ldots,p_m$.  To date, the application of Stein's method to this problem has been fairly limited.  To the best of our knowledge, the only work  thus far on this topic are the unpublished papers \cite{mann} and \cite{mann2}.  The first of these papers \cite{mann} uses an exchangeable pair coupling to study the asymptotic properties of the statistic with a uniform null distribution using a smooth test function, and in it is derived a bound of order $n^{-1/2}$, but with an unmanageably large constant.  In  the second paper \cite{mann2}, a bound (Theorem 1.3) of  \cite{gotze} for the multidimensional central limit theorem is used to derived a bound, in the Kolmogorov distance, on Pearson's statistic with general null distribution:
\begin{equation}\label{csdbcbhj}\sup_{z>0}|\mathbb{P}(W\leq z)-\mathbb{P}(Y_{m-1}\leq z)|\leq\frac{250m}{p_*^{3/2}\sqrt{n}},
\end{equation}
where $Y_{m-1}\sim \chi_{(m-1)}^2$ and $p_*=\min_{1\leq i\leq m}p_i$.  The dependence on $m$ in this bound can actually be improved by using Theorem 1.1 of \cite{bentkus} in place of Theorem 1.3 of \cite{gotze}, which yields the upper bound $400m^{1/4}p_*^{-3/2}n^{-1/2}$.

An explicit $O(n^{-1/2})$ bound for the distance between Pearson's statistic and its limiting chi-square distribution is given by (\ref{csdbcbhj}).  However, this rate is not optimal.   \cite{yarnold} used Edgeworth expansions to show that
\begin{equation*}\sup_{z>0}|\mathbb{P}(W\leq z)-\mathbb{P}(Y_{m-1}\leq z)|=O(n^{-(m-1)/m}), \quad m\geq 2,
\end{equation*}
and this was improved, for $m\geq 6$, by  \cite{gu03}:
\begin{equation*}\sup_{z>0}|\mathbb{P}(W\leq z)-\mathbb{P}(Y_{m-1}\leq z)|=O(n^{-1}), \quad m\geq 6.
\end{equation*}
However, \cite{yarnold} and \cite{gu03} do not give explicit upper bounds.  It should also be noted that  \cite{ulyanov} and \cite{asylbekov} have used Edgeworth expansions to study the rate of convergence of the more general power divergence family of statistics (see \cite{cr84}) constructed from the multinomial distribution of degree $m$ (which includes Pearson's statistic, the log-likelihood ratio statistic and Freeman-Tukey statistics as special cases) to their $\chi_{(m-1)}^2$ limits.

In this section, we obtain explicit bounds for the distributional distance between Pearson's statistic and its limiting chi-square distribution.  We now present our bounds, starting with one of the main results of this paper: a Pearson chi-square weak convergence theorem
for smooth test functions with a bound of order $n^{-1}$, which holds for all $m\geq2$.  

\begin{theorem}\label{Pearson_thm_non_int}
Let $(U_1, \ldots, U_m)$ represent the multinomial vector of $n\geq 2$ observed counts, where $m \geq 2$, and suppose that $np_j\geq 1$ for all $j=1,\ldots,m$.  Denote the Pearson statistic  by $W$. Let $h \in C_b^5(\mathbb{R}^+)$. 
 Then
\begin{align}\label{pearbound}
|\mathbb{E} h(W) - \chi^2_{(m-1)} h | &\leq \frac{4}{(m+1)n}\bigg(\sum_{j=1}^m\frac{1}{\sqrt{p_j}}\bigg)^2\{19\|h\|+366\|h'\|+2016\|h''\|\nonumber\\
&\quad+5264\|h^{(3)}\|+106965\|h^{(4)}\|+302922\|h^{(5)}\|\}.
\end{align}
\end{theorem}

When the constants are large compared to $n$ then the next result may give the smaller numerical bound. 

\begin{theorem}\label{pearsqrtn}Suppose $np_j\geq 1$ for all $j=1,\ldots,m$.  Then, for $h \in C_b^2(\mathbb{R}^+)$,
\begin{equation}\label{pearbound0}
|\mathbb{E} h(W) - \chi^2_{(m-1)} h | \leq \frac{12}{(m+1)\sqrt{n}}\{6\|h\|+46\|h'\|+84\|h''\|\}\sum_{j=1}^m\frac{1}{\sqrt{p_j}}.
\end{equation}
\end{theorem}

Let $p_*=\min_{1\leq i\leq m}p_i$; 
using that 
$\sum_{j=1}^m\frac{1}{\sqrt{p_j}}\le \frac{m}{\sqrt{p_*}}$ the next corollary is immediate from \eqref{pearbound} and \eqref{pearbound0}.

\begin{corollary} Suppose that $np_*\geq1$.  Then, for $h \in C_b^2(\mathbb{R}^+)$,
\begin{equation}\label{pearbound01}
|\mathbb{E} h(W) - \chi^2_{(m-1)} h | \leq \frac{12}{\sqrt{np_*}}\{6\|h\|+46\|h'\|+84\|h''\|\}
\end{equation}
and, for $h \in C_b^5(\mathbb{R}^+)$,
\begin{align}\label{pearbound1}
|\mathbb{E} h(W) - \chi^2_{(m-1)} h | &\leq \frac{4m}{np_*}\{19\|h\|+366\|h'\|+2016\|h''\|\nonumber\\
&\quad+5264\|h^{(3)}\|+106965\|h^{(4)}\|+302922\|h^{(5)}\|\}.
\end{align}
\end{corollary}

Applying a basic technique for converting smooth test function bounds into Kolmogorov distance bounds, which can be found in \cite{chen}, p$.$ 48, to bound (\ref{pearbound01}) gives  Kolmogorov distance bounds for the rate of convergence of Pearson's statistic.  The standard proof is given in Section 5.

\begin{corollary}\label{kolcol}Let $Y_d$ denote a $\chi_{(d)}^2$ random variable and suppose $np_*\geq1$.  Then
\begin{align*}&\sup_{z>0}|\mathbb{P}(W\leq z)-\mathbb{P}(Y_{m-1}\leq z)| \\
&\leq  \begin{cases} \displaystyle \frac{1}{(np_*)^{1/10}}\bigg\{8+\frac{21}{(np_*)^{1/5}}+\frac{72}{(np_*)^{2/5}}\bigg\}, & \:  m=2, \\
\displaystyle \frac{1}{(np_*)^{1/6}}\bigg\{19+\frac{44}{(np_*)^{1/6}}+\frac{72}{(np_*)^{1/3}}\bigg\}, & \:  m=3, \\
\displaystyle \frac{1}{(m-3)^{1/3}(np_*)^{1/6}}\bigg\{13+\frac{37(m-3)^{1/6}}{(np_*)^{1/6}}+\frac{72(m-3)^{1/3}}{(np_*)^{1/3}}\bigg\}, & \:   m\geq4. \end{cases}
\end{align*}
\end{corollary}

\begin{remark}The assumption $np_*\geq 1$ is very mild (bounds (\ref{pearbound}) and (\ref{pearbound0}) are uninformative otherwise) and is included for the sole purpose of simplifying calculations.  Indeed, we do not claim that the numerical constants in our bounds are close to optimal.  In order to simplify the calculations and obtain a compact final bound, we make a number of simple and rather crude approximations.   We do, however, take care to ensure that these approximations do not affect the role of $m$, $n$ and $p_1,\ldots, p_m$ in our final bounds.
\end{remark}

\begin{remark}For fixed $m$, bounds (\ref{pearbound01}) and (\ref{pearbound1}) depend on $n$ and $p_*$ in the correct way, in the sense that they tend to zero if and only if $np_*\rightarrow\infty$.  This is  an established condition under which Pearson's statistic $W$ converges to the $ \chi^2_{(m-1)}$ distribution; see  \cite{greenwood}.  However, if we allow $m$ to vary with $n$, then Theorem \ref{pearsqrtn} shows that this is not a necessary condition; some cell probabilities may be of the order $n^{-1}$ as long as $m$ is large and the bound could still be small. To illustrate this point, assume that $p_1 = n^{-1}$ and for $j=2, \ldots, m$, $p_j = \frac{n-1}{n(m-1)}.$ Then $\sum_{i=1}^m p_i = 1$ and the bound in  Theorem \ref{pearsqrtn} is of order $\frac{1}{m} + \sqrt{\frac{m}{n}}$. As long as $m =m(n) \rightarrow \infty$ as $n \rightarrow \infty$ such that $m/n\rightarrow 0$ as $n \rightarrow \infty$ the chi-square approximation is valid.
\end{remark}

\begin{remark}Considering the situation where $m\rightarrow\infty$ as $n\rightarrow\infty$ more generally,   \cite{koehler} give an example of when this might occur in contingency table analysis by considering cross-classification of discrete variables.  \cite{tumanyan} showed that the null distribution of the Pearson statistic is asymptotically normal if $np_*\rightarrow\infty$ and $m\rightarrow\infty$.  Note that this result is consistent with the chi-square convergence result since the chi-square distribution converges to the normal as the degrees of freedom increase.  An attractive property of bound (\ref{pearbound01}) is that it does not involve $m$, meaning that it tends to zero if $np_*\rightarrow\infty$ and $m\rightarrow\infty$.  However, bound (\ref{pearbound1}) does involve $m$ and as a result tends to zero only if the stronger condition $np_*/m\rightarrow\infty$ holds.  This, along with the very large numerical constants indicates a weakness in the bound; the price we pay for the faster $n^{-1}$ convergence rate.
\end{remark}

\begin{remark}The Kolmogorov distance bound of Corollary \ref{kolcol} has a suboptimal dependence on $n$.  This is perhaps to be expected, as the bound was derived by applying a fairly crude non-smooth test function approximation technique to the smooth test function bound (\ref{pearbound01}).  However, it is the first bound in the literature for the rate of convergence in Kolmogorov distance of Pearson's statistic that depends on $n$ and $p_1,\ldots,p_m$ in the correct manner, in the sense that it tends to zero if and only if $np_*\rightarrow\infty$.  In fact, to the best of our knowledge there does not exist bound in any metric for the rate of convergence of Pearson's statistic that depends on $n$ and $p_1,\ldots,p_m$ in this manner.  

This demonstrates the power of Stein's method, which allows us to use the multivariate normal Stein equation to effectively deal with the dependence structure of Pearson's statistic (see the proofs of Theorems \ref{Pearson_thm_non_int} and \ref{pearsqrtn}).  A direction for future research would be to use Stein's method, with an approach tailored to non-smooth test functions, to Pearson's statistic.  With this approach it may be possible to obtain a bound which tends to zero if any only if $np_*\rightarrow\infty$ but with a better rate in $n$.
\end{remark}

Let us now prove Theorem \ref{Pearson_thm_non_int}; the proof of Theorem \ref{pearsqrtn} is much simpler and is given in Section 5.

\vspace{3mm}

\noindent \emph{Proof of Theorem \ref{Pearson_thm_non_int}.}  As was the case for the proof of Theorem \ref{chiiid}, the proof comes in two parts.  The first part includes expansions and bounds,
and the second part includes the symmetry argument.

\medskip
\noindent{\bf{Proof Part I: Expansions and bounding}}

\medskip 
 Let $S_j = \frac{1}{\sqrt{np_j}}(U_j - np_j)$, so that $S_j$ denotes the standardised cell counts, and notice that $\sum_{j=1}^m \sqrt{p_j}S_j = 0$; also $U_j \sim \mathrm{Bin}(n, p_j)$ for each $j$. 
Now, 
\begin{equation*} W =  \sum_{j=1}^m S_j^2.  \end{equation*}  
This is analogous to the already studied case of independent summands in a chi-square statistic. 
 The key difference in this application is that the $S_j$ are \emph{not} independent, although they can be 
constructed from independent indicators.  Let $I_j(i)$ be the indicator that trial $i$ results in classification 
in cell $j$, and let $\tilde{I_j}(i) = I_j(i) - p_j$ be its standardised version.  Then it is clear that $S_j = \frac{1}{\sqrt{np_j}}\sum_{i=1}^n \tilde{I}_j(i)$.  

To take advantage of the independence of the indicators $I_j(1), \ldots, I_j(n)$, we define $S_j^{(i)}=S_j-\frac{1}{\sqrt{np_j}}I_j(i)$ and denote the vector $(S_1^{(i)},\ldots,S_{m}^{(i)})$ by $\mathbf{S}^{(i)}$.  Note that $\mathbf{S}^{(i)}$ is independent of $I_j(i)$ for $i=1,\ldots,n$.  The following expressions for the moments of the $S_j^{(i)}$ are straightforward:
\begin{lemma}\label{thismoment}The second, fourth and sixth moments of $S_j^{(i)}$ are given by
\begin{eqnarray*}\mathbb{E}(S_j^{(i)})^2&=&\frac{(n-1)(1-p_j)}{n}+\frac{p_j}{n}, \\
\mathbb{E}(S_j^{(i)})^4&=&3(1-p_j)^2\frac{(n-1)}{n}+\frac{(n-1)}{n^2p_j}(1-p_j)(1-13p_j+23p_j^2)+\frac{p_j^2}{n^2}, \\
\mathbb{E}(S_j^{(i)})^6&=& 15(1-p_j)^3+\frac{P_1(p_j)}{np_j}+\frac{P_2(p_j)}{(np_j)^2}+\frac{P_3(p_j)}{(np_j)^2n},
\end{eqnarray*}
where
\begin{eqnarray*}P_1(p_j)&=&(1-p_j)(1-87p_j+724p_j^2-1626p_j^3+1044p_j^4),\\
P_2(p_j)&=&5(1-p_j)^2(5-47p_j+68p_j^2),\\
P_3(p_j)&=&-(1-2p_j)(1-60p_j+420p_j^2-720p_j^3+360p_j^4).
\end{eqnarray*}

If $np_j\geq 1$, then $\mathbb{E}|S^{(i)}|<1$, $\mathbb{E}(S_j^{(i)})^2<1$, $\mathbb{E}|S_j^{(i)}|^3<4^{3/4}$, $\mathbb{E}(S_j^{(i)})^4<4$, $\mathbb{E}|S_j^{(i)}|^5<42^{5/6}$ and $\mathbb{E}(S_j^{(i)})^6<42$.
\end{lemma}

\begin{proof}The equalities can be obtained by using the formulas for the first six moments of the Binomial distribution, which are given in \cite{stuart}.  To obtain the inequalities for the fourth and sixth moments, note that in the interval $[0,1]$ the functions $g_1(p)=3(1-p)^2+(1-p)|1-13p+23p^2|+p^2$ and $g_2(p)=15(1-p)^3+|P_1(p)|+|P_2(p)|+|P_3(p)|$ take their maximum value at $p=0$.  The inequalities for the first, third and fifth absolute moments follow from applying H\"{o}lder's inequality to the  inequalities for the second, fourth and sixth moments, respectively.
\end{proof}

We shall take a similar approach to before, where we converted the $\chi^2_{(1)}$ Stein equation into the $N(0,1)$ Stein equation; this time, however, we shall convert to a multivariate normal, which has Stein equation (see, for example  \cite{goldstein1}):
\begin{equation}\label{mvn1} \nabla^T\Sigma\nabla f(\mathbf{s})-\mathbf{s}^T\nabla f(\mathbf{s})=h(\mathbf{s})-\mathbb{E}h(\mathbf{Z}),
\end{equation}
where $\mathbf{Z}\sim \mathrm{MVN}(\mathbf{0},\Sigma)$.  In analogy with \eqref{normalop} we define the Stein operator 
\begin{equation}\label{multinormalop}
\AAA_{\mathrm{MVN}(\mathbf{0},\Sigma)} f(\bs) =  \nabla^T\Sigma\nabla f(\mathbf{s})-\mathbf{s}^T\nabla f(\mathbf{s}) .
\end{equation} 
Since for  $\mathbf{S} = (S_1, \ldots, S_m)$, the covariance matrix of $\bS$ is  $\Sigma_{\mathbf{S}} = \Sigma_\mathbf{U}$, where $\Sigma_\mathbf{U}$ denotes the covariance matrix of $\mathbf{U}$, it follows
by a simple random sampling without replacement argument (see \cite{rice}, Section 7.3) that 
$\Sigma_{\mathbf{S}} = (\sigma_{jk})$, has entries
\begin{equation*}
 \sigma_{jj} = 1-p_j \quad \text{and} \quad \sigma_{jk} = -\sqrt{p_jp_k}, \quad j\not=k. \end{equation*}

The following connection between the $\chi^2_{(m-1)}$ and $\mathrm{MVN}(\mathbf{0}, \Sigma_{\mathbf{S}})$ Stein equations will be proven  in Section 5.

\begin{lemma} \label{operatorcomp} Let $\mathcal{A}_{\mathrm{MVN}(\mathbf{0},\Sigma_\mathbf{S})}$ be given in \eqref{multinormalop} and let $\AAA_{m-1}$ be given in \eqref{chiop}.  Let  $f \in C^2(\mathbb{R})$ and define $g: \mathbb{R}^m \to \mathbb{R}$ by $g(\mathbf{s}) = \frac{1}{4}f(w)$ with $w = \sum_{i=1}^m s_i^2$ for $\bs = (s_1, \ldots, s_m)$. If $\sum_{j=1}^m \sqrt{p_j}s_j = 0$ then 
\begin{equation*} \mathcal{A}_{\mathrm{MVN}(\mathbf{0},\Sigma_{\mathbf{S}})} g(\mathbf{s}) =  \mathcal{A}_{m-1} f(w). 
\end{equation*}
\end{lemma}

We wish to bound $\mathbb{E}\mathcal{A}_{m-1}g(W)$, which by Lemma \ref{operatorcomp}  is equivalent to bounding $\mathbb{E}\mathcal{A}_{\mathrm{MVN}(\mathbf{0},\Sigma_\mathbf{S})}g(\mathbf{S})$.  As the indicators $I_j(1),I_j(2), \ldots, I_j(n)$ are identically distributed, it follows that
\[\sum_{j=1}^{m}\mathbb{E}S_j\frac{\partial g}{\partial s_j}(\mathbf{S})=\frac{1}{\sqrt{n}}\sum_{i=1}^n\sum_{j=1}^{m}\frac{1}{\sqrt{p_j}}\mathbb{E}\tilde{I}_j(i)\frac{\partial g}{\partial s_j}(\mathbf{S})=\sqrt{n}\sum_{j=1}^{m}\frac{1}{\sqrt{p_j}}\mathbb{E}\tilde{I}_j(1)\frac{\partial g}{\partial s_j}(\mathbf{S}).\]
We now Taylor expand and use the independence of $\mathbf{S}^{(1)}$ and the $I_j(1)$ to obtain
\begin{align*}\sum_{j=1}^m\mathbb{E}S_j\frac{\partial g}{\partial s_j}(\mathbf{S})&=\sqrt{n}\sum_{j=1}^m\frac{1}{\sqrt{p_j}}\mathbb{E}\tilde{I}_j(1)\mathbb{E}\frac{\partial g}{\partial s_j}(\mathbf{S}^{(1)})\\
&\quad+\sum_{j=1}^m\sum_{k=1}^m\frac{1}{\sqrt{p_jp_k}}\mathbb{E}\tilde{I}_j(1)I_k(1)\mathbb{E}\frac{\partial^2g}{\partial s_j\partial s_k}(\mathbf{S}^{(1)}) +N_1+R_1,
\end{align*}
where
\begin{align*}N_1&=\frac{1}{2\sqrt{n}}\sum_{j=1}^m\sum_{k=1}^m\sum_{l=1}^{m}\frac{1}{\sqrt{p_jp_kp_l}}\mathbb{E}\tilde{I}_j(1)I_k(1)I_l(1)\mathbb{E}\frac{\partial^3g}{\partial s_j\partial s_k\partial s_l}(\mathbf{S}^{(1)}),\\
R_1&=\frac{1}{6n}\sum_{j=1}^m\sum_{k=1}^m\sum_{l=1}^{m}\sum_{t=1}^{m}\frac{1}{\sqrt{p_jp_kp_lp_t}}\mathbb{E}\tilde{I}_j(1)I_k(1)I_l(1)I_t(1)\frac{\partial^4g}{\partial s_j\partial s_k\partial s_l\partial s_t}(\boldsymbol{\xi}).
\end{align*}
Here and throughout the proof $\boldsymbol{\xi}$ denotes a vector with $m$ entries with $j$-th entry $\xi_j=S_j^{(1)}+\frac{\theta_j}{\sqrt{np_j}}I_j(1)$ for some $\theta_j\in(0,1)$.   Using that $\mathbb{E}\tilde{I}_j=0$ and then Taylor expanding gives
\begin{align*}\sum_{j=1}^m\mathbb{E}S_j\frac{\partial g}{\partial s_j}(\mathbf{S})&=\sum_{j=1}^m\sum_{k=1}^m\frac{1}{\sqrt{p_jp_k}}\mathbb{E}\tilde{I}_j(1)I_k(1)\mathbb{E}\frac{\partial^2g}{\partial s_j\partial s_k}(\mathbf{S}^{(1)}) \!+\!N_1\!+\!R_1 \\
&=\sum_{j=1}^m\sum_{k=1}^m\frac{1}{\sqrt{p_jp_k}}\mathbb{E}\tilde{I}_j(1)I_k(1)\mathbb{E}\frac{\partial^2g}{\partial s_j\partial s_k}(\mathbf{S}) \!+\!N_1\!+\!N_2\!+\!R_1\!+R_2,
\end{align*}
where
\begin{align*}N_2&=-\frac{1}{\sqrt{n}}\sum_{j=1}^m\sum_{k=1}^m\sum_{l=1}^{m}\frac{1}{\sqrt{p_jp_kp_l}}\mathbb{E}\tilde{I}_j(1)I_k(1)\mathbb{E}I_l(1)\frac{\partial^3g}{\partial s_j\partial s_k\partial s_l}(\mathbf{S}),\\
R_2&=-\frac{1}{2n}\sum_{j=1}^m\sum_{k=1}^m\sum_{l=1}^{m}\sum_{t=1}^{m}\frac{1}{\sqrt{p_jp_kp_lp_t}}\mathbb{E}\tilde{I}_j(1)I_k(1)\mathbb{E}I_l(1)I_t(1)\frac{\partial^4g}{\partial s_j\partial s_k\partial s_l\partial s_t}(\boldsymbol{\xi}),
\end{align*}

Note that the vector $\boldsymbol{\xi}$ in the above formula for $R_2$ is in general different from the $\boldsymbol{\xi}$ in the expression for $R_1$, as was the case in the proof of Theorem \ref{chiiid}.  As  $\tilde{I}_j(i)I_j(i)=(1-p_j)I_j(i)$ and, for $j\not=k$, $\tilde{I}_j(i)I_k(i)=(I_j(i)-p_j)I_k(i)=-p_jI_k(i)$, since each trial leads to a unique classification,  taking expectations gives
\begin{equation}
\label{crossmeans} \mathbb{E}\tilde{I}_j(i)I_j(i)=p_j(1-p_j) \quad \mbox{and} \quad \mathbb{E}\tilde{I}_j(i)I_k(i)=-p_jp_k, \: j\not=k,
\end{equation} 
and so
\[\frac{1}{p_j}\mathbb{E}\tilde{I}_j(i)I_j(i)=1-p_j=\sigma_{jj}; \quad \frac{1}{\sqrt{p_jp_k}}\mathbb{E}\tilde{I}_j(i)I_k(i)=-\sqrt{p_jp_k}=\sigma_{jk}, \: j\not=k.\]
Therefore
\[\mathbb{E}\bigg[\sum_{j=1}^m\sum_{k=1}^m\sigma_{jk}\frac{\partial^2g}{\partial s_j\partial s_k}(\mathbf{S})-\sum_{j=1}^mS_j\frac{\partial g}{\partial s_j}(\mathbf{S})\bigg]=-N_1-N_2-R_1-R_2,\]
and it remains to bound the terms $N_1$, $N_2$, $R_1$ and $R_2$.  

Before bounding these terms, we deal with $N_2$.  Taylor expanding $\frac{\partial^3g}{\partial s_j\partial s_k\partial s_l}(\mathbf{S})$ about $\mathbf{S}^{(1)}$ yields
\[N_2=N_3+R_3,\]
where
\begin{align*}N_3&=-\frac{1}{\sqrt{n}}\sum_{j=1}^m\sum_{k=1}^m\sum_{l=1}^{m}\frac{1}{\sqrt{p_jp_kp_l}}\mathbb{E}\tilde{I}_j(1)I_k(1)\mathbb{E}I_l(1)\mathbb{E}\frac{\partial^3g}{\partial s_j\partial s_k\partial s_l}(\mathbf{S}^{(1)})
\end{align*}
and
\begin{align*}
R_3&=\frac{1}{n}\sum_{j=1}^m\sum_{k=1}^m\sum_{l=1}^{m}\sum_{t=1}^{m}\frac{1}{\sqrt{p_jp_kp_lp_t}}\mathbb{E}\tilde{I}_j(1)I_k(1)\mathbb{E}I_l(1)I_t(1)\frac{\partial^4g}{\partial s_j\partial s_k\partial s_l\partial s_t}(\boldsymbol{\xi}).
\end{align*}

We can immediately bound $R_1$, $R_2$ and $R_3$ to the desired order of $O(n^{-1})$, but a more detailed calculation, involving symmetry arguments, is required to bound $N_1$ and $N_3$ to this order.  We begin by bounding $R_1$.  Recalling that $\tilde{I}_j(i)I_j(i)=(1-p_j)I_j(i)$, $\tilde{I}_j(i)I_k(i)=-p_jI_k(i)$ for $j\not=k$ and $I_j(i)I_k(i)=0$ for $j\not=k$, 
\begin{align}|R_1|&=\frac{1}{6n}\bigg|\sum_{j=1}^m\frac{1-p_j}{p_j^2}\mathbb{E}I_j(1)\frac{\partial^4g}{\partial s_j^4}(\boldsymbol{\xi})-\sum_{j=1}^m\sum_{k\not=j}^m\frac{\sqrt{p_j}}{p_k^{3/2}}\mathbb{E}I_k(1)\frac{\partial^4 g}{\partial s_j\partial s_k^3}(\boldsymbol{\xi})\bigg| \nonumber \\
\label{r11111}&\leq\frac{1}{6n}\bigg\{\sum_{j=1}^m\frac{1}{p_j^2}\mathbb{E}\bigg|I_j(1)\frac{\partial^4g}{\partial s_j^4}(\boldsymbol{\xi})\bigg|+\sum_{j=1}^m\sum_{k\not=j}^m\frac{\sqrt{p_j}}{p_k^{3/2}}\mathbb{E}\bigg|I_k(1)\frac{\partial^4 g}{\partial s_j\partial s_k^3}(\boldsymbol{\xi})\bigg|\bigg\}.
\end{align}
To obtain the desired $O(n^{-1})$ rate for $R_1$, we need to show that the two expectations given in (\ref{r11111}) are $O(1)$.  This is somewhat involved, and is deferred until we have bounds of a similar form to (\ref{r11111}) for $R_2$ and $R_2$.  The terms $R_2$ and $R_3$ cam be bounded using a similar approach,
\begin{align*}|R_2|+|R_3| &\leq \frac{3}{2n}\sum_{j=1}^m\sum_{k=1}^m\sum_{l=1}^m\frac{1}{\sqrt{p_jp_kp_l^2}}|\mathbb{E}\tilde{I}_j(1)I_k(1)|\mathbb{E}\bigg|I_l(1)\frac{\partial^4g}{\partial s_j\partial s_k\partial s_l^2}(\boldsymbol{\xi})\bigg| \\
&= \frac{3}{2n}\bigg\{\sum_{j=1}^m\sum_{l=1}^{m}\frac{1-p_j}{p_l}\mathbb{E}\bigg|I_l(1)\frac{\partial^4g}{\partial s_j^2\partial s_l^2}(\boldsymbol{\xi})\bigg|  \\
&\quad+\sum_{j=1}^m\sum_{k\not=j}^m\sum_{l= 1}^{m}\frac{\sqrt{p_jp_k}}{p_l}\mathbb{E}\bigg|I_l(1)\frac{\partial^4g}{\partial s_j\partial s_k\partial s_l^2}(\boldsymbol{\xi})\bigg|\bigg\}.
\end{align*}
Collecting the bounds for $R_1$, $R_2$ and $R_3$ and then using that $0<p_j<1$, $j=1,\ldots,m$, to simplify the resulting bound gives 
\begin{align}|R_1|\!+\!|R_2|\!+\!|R_3|&\leq\frac{1}{6n}\bigg\{\!\sum_{j=1}^m\frac{1}{p_j^2}\mathbb{E}\bigg|I_j(1)\frac{\partial^4g}{\partial s_j^4}(\boldsymbol{\xi})\bigg| \!+\!9\sum_{j=1}^m\sum_{k=1}^m\frac{1}{p_j}\mathbb{E}\bigg|I_j(1)\frac{\partial^4g}{\partial s_j^2\partial s_k^2}(\boldsymbol{\xi})\bigg| \nonumber \\
&\quad+\sum_{j=1}^m\sum_{k\not=j}^m\frac{\sqrt{p_k}}{p_j^{3/2}}\mathbb{E}\bigg|I_j(1)\frac{\partial^4g}{\partial s_j^3\partial s_k}(\boldsymbol{\xi})\bigg|\nonumber \\
\label{re123z}&\quad+9\sum_{j=1}^m\sum_{k\not=j}^m\sum_{l=1}^m\frac{\sqrt{p_kp_l}}{p_j}\mathbb{E}\bigg|I_j(1)\frac{\partial^4g}{\partial s_j^2\partial s_k\partial s_l}(\boldsymbol{\xi})\bigg|\bigg\}.
\end{align}

To bound the expectations on the right-hand side of (\ref{re123z}), straightforward differentiation gives 
\[\frac{\partial^4g}{\partial s_j^4}(\mathbf{s})=3f''(w)+12s_j^2f^{(3)}(w)+4s_j^4f^{(4)}(w),\]
and similar expressions hold for mixed partial derivatives. 
Hence, for any $j,k,l,t$, 
\begin{equation}\label{sebdhd}\bigg|\frac{\partial^4g}{\partial s_j\partial s_k\partial s_l\partial s_t}(\mathbf{s})\bigg|\leq 3\|f''\|+3(s_j^2+s_k^2+s_l^2+s_t^2)\|f^{(3)}\|+(s_j^4+s_k^4+s_l^4+s_t^4)\|f^{(4)}\|.
\end{equation}

The following lemma is proved in Section 5.  Some of the inequalities given in this lemma are only used in the proof of Theorem \ref{pearsqrtn} but are collected here for completeness.
\begin{lemma}\label{qqqppp}For all $j,k=1,\ldots,m$, we have $\mathbb{E}|I_j(1)\xi_k|< 2p_j$, $\mathbb{E}I_j(1)\xi_k^2<4p_j$, $\mathbb{E}|I_j(1)\xi_k^3|<14p_j$, $\mathbb{E}I_j(1)\xi_k^4<27p_j$ and $\mathbb{E}I_j(1)\xi_k^6<305p_j$.
\end{lemma}

Using inequality (\ref{sebdhd}), Lemma \ref{qqqppp} and the triangle inequality,  the expectations on the right-hand side of (\ref{re123z}) can be bounded as follows;
\begin{align}&\mathbb{E}\bigg|I_j(1)\frac{\partial^4g}{\partial s_j\partial s_k\partial s_l\partial s_t}(\boldsymbol{\xi})\bigg|\nonumber \\
&\leq 3\|f''\|\mathbb{E}I_j(1)+3\|f^{(3)}\|\Big[\mathbb{E}I_j(1)\xi_j^2+\mathbb{E}I_j(1)\xi_k^2+\mathbb{E}I_j(1)\xi_l^2+\mathbb{E}I_j(1)\xi_t^2\Big] \nonumber \\
&\quad+\|f^{(4)}\|\Big[\mathbb{E}I_j(1)\xi_j^4+\mathbb{E}I_j(1)\xi_k^4+\mathbb{E}I_j(1)\xi_l^4+\mathbb{E}I_j(1)\xi_t^4 \Big] \nonumber\\
&\leq p_j\{3\|f''\|+48\|f^{(3)}\|+108\|f^{(4)}\|\}. \label{f4ineqpsi}
\end{align}
Applying this bound gives the desired $O(n^{-1})$ rate for the remainder terms $R_1$, $R_2$, $R_3$:
\begin{align*}|R_1|+|R_2|+|R_3|&\leq \frac{1}{6n}\{3\|f''\|+48\|f^{(3)}\|+108\|f^{(4)}\|\}\bigg\{\sum_{j=1}^m\frac{1}{p_j}+9m^2\\
&\quad+\sum_{j=1}^m\sum_{k\not=j}^m\sqrt{\frac{p_k}{p_j}}+9\sum_{j=1}^m\sum_{k\not=j}^m\sum_{l=1}^m\sqrt{p_kp_l}\bigg\} \\
&\leq\frac{1}{n}\{10\|f''\|+160\|f^{(3)}\|+360\|f^{(4)}\|\}\sum_{j=1}^m\frac{1}{p_j}.
\end{align*}

Recalling that $I_j(i)I_k(i)=0$ if $j\not=k$ and \eqref{crossmeans},
\begin{align*}&N_1+N_3\\
&=\frac{1}{2\sqrt{n}}\sum_{j=1}^m\sum_{k=1}^m\frac{1}{\sqrt{p_jp_k^2}}\mathbb{E}\tilde{I}_j(1)I_k(1)\mathbb{E}\frac{\partial^3g}{\partial s_j\partial s_k^2}(\mathbf{S}^{(1)}) \\
&\quad -\frac{1}{\sqrt{n}}\sum_{j=1}^m\sum_{k=1}^m\sum_{l=1}^m\frac{\mathbb{E}\tilde{I}_j(1)I_k(1)}{\sqrt{p_jp_kp_l}}\cdot p_l\mathbb{E}\frac{\partial^3g}{\partial s_j\partial s_k\partial s_l}(\mathbf{S}^{(1)})\\
&=\frac{1}{2\sqrt{n}}\bigg\{\sum_{j=1}^m\frac{1}{\sqrt{p_j}}(1-p_j)(1-2p_j)\mathbb{E}\frac{\partial^3g}{\partial s_j^3}(\mathbf{S}^{(1)})-\sum_{j=1}^m\sum_{k\not=j}^m\sqrt{p_j}\mathbb{E}\frac{\partial^3g}{\partial s_j\partial s_k^2}(\mathbf{S}^{(1)})\\
&\quad-2\sum_{j=1}^m\sum_{l=1}^m(1-p_j)\sqrt{p_l}\mathbb{E}\frac{\partial^3g}{\partial s_j^2\partial s_l}(\mathbf{S}^{(1)})\\
&\quad+2\sum_{j=1}^m\sum_{k\not=j}^m\sum_{l=1}^m\sqrt{p_jp_kp_l}\mathbb{E}\frac{\partial^3g}{\partial s_j\partial s_k\partial s_l}(\mathbf{S}^{(1)})\bigg\}.
\end{align*}
By Taylor expanding in the usual manner, 
\[N_1+N_3=N_4+R_4,\]
where
\begin{align*}N_4&=\frac{1}{2\sqrt{n}}\bigg\{\sum_{j=1}^m\frac{1}{\sqrt{p_j}}(1-p_j)(1-2p_j)\mathbb{E}\frac{\partial^3g}{\partial s_j^3}(\mathbf{S})-\sum_{j=1}^m\sum_{k\not=j}^m\sqrt{p_j}\mathbb{E}\frac{\partial^3g}{\partial s_j\partial s_k^2}(\mathbf{S})\\
&\quad-\!2\sum_{j=1}^m\sum_{l=1}^m\!(1-p_j)\sqrt{p_l}\mathbb{E}\frac{\partial^3g}{\partial s_j^2\partial s_l}(\mathbf{S})\!+\!2\sum_{j=1}^m\sum_{k\not=j}^m\sum_{l=1}^m\!\sqrt{p_jp_kp_l}\mathbb{E}\frac{\partial^3g}{\partial s_j\partial s_k\partial s_l}(\mathbf{S})\bigg\},\\
|R_4|&\leq  \frac{1}{2n}\bigg\{\!\sum_{j=1}^m\sum_{k=1}^m\!\frac{1}{\sqrt{p_j}}\mathbb{E}\bigg|I_k(1)\frac{\partial^4g}{\partial s_j^3\partial s_k}(\boldsymbol{\xi})\bigg|\!+\!\sum_{j=1}^m\sum_{k\not=j}^m\sum_{l=1}^m\!\sqrt{p_j}\mathbb{E}\bigg|I_l(1)\frac{\partial^4g}{\partial s_j\partial s_k^2\partial s_l}(\boldsymbol{\xi})\bigg|\\
&\quad+2\sum_{j=1}^m\sum_{l=1}^m\sum_{t=1}^m\sqrt{p_l}\mathbb{E}\bigg|I_t(1)\frac{\partial^4g}{\partial s_j^2\partial s_l\partial s_t}(\boldsymbol{\xi})\bigg|\\
&\quad+2\sum_{j=1}^m\sum_{k\not=j}^m\sum_{l=1}^m\sum_{t=1}^m\sqrt{p_jp_kp_l}\mathbb{E}\bigg|I_t(1)\frac{\partial^4g}{\partial s_j\partial s_k\partial s_l\partial s_t}(\boldsymbol{\xi})\bigg|\bigg\}.
\end{align*}

We can bound $R_4$ by applying inequality (\ref{f4ineqpsi}):
\begin{align*}|R_4|&\leq\frac{1}{2n}\{3\|f''\|+48\|f^{(3)}\|+108\|f^{(4)}\|\}\bigg\{\sum_{j=1}^m\sum_{k=1}^m\frac{p_k}{\sqrt{p_j}}+\sum_{j=1}^m\sum_{k\not=j}^m\sum_{l=1}^m\sqrt{p_j}p_l\\
&\quad+2\sum_{j=1}^m\sum_{l=1}^m\sum_{t=1}^m\sqrt{p_l}p_t+2\sum_{j=1}^m\sum_{k\not=j}^m\sum_{l=1}^m\sum_{t=1}^m\sqrt{p_jp_kp_l}p_t\bigg\}\\
&\leq\frac{1}{n}\{9\|f''\|+144\|f^{(3)}\|+324\|f^{(4)}\|\}\sum_{j=1}^m\frac{1}{\sqrt{p_j}}.
\end{align*}

For $N_4$, by writing partial derivatives of $g$ in terms of derivatives of $f$, 
\begin{align*}N_4&=
\frac{1}{2\sqrt{n}}\bigg\{\sum_{j=1}^m\frac{1}{\sqrt{p_j}}(1-p_j)(1-2p_j)\{3\mathbb{E}S_jf''(W)+2\mathbb{E}S_j^3f^{(3)}(W)\}\\
&\quad-\sum_{j=1}^m\sum_{k\not=j}^m\sqrt{p_j}\{\mathbb{E}S_jf''(W)+2\mathbb{E}S_jS_k^2f^{(3)}(W)\}\\
&\quad-2\sum_{j=1}^m\sum_{l=1}^m(1-p_j)\sqrt{p_l}\{\mathbb{E}S_lf''(W)+2\mathbb{E}S_j^2S_lf^{(3)}(W)\}\\
&\quad+4\sum_{j=1}^m\sum_{k\not=j}^m\sum_{l=1}^m\sqrt{p_jp_kp_l}\mathbb{E}S_jS_kS_lf^{(3)}(W)\bigg\}.
\end{align*}
Since $\sum_{j=1}^m\sqrt{p_j}S_j=0$, the final two sums in the above display  equal  $0$.  Furthermore, $\sum_{j=1}^m\sum_{k\not=j}^m\sqrt{p_j}\mathbb{E}S_jS_k^2f^{(3)}(W)=-\sum_{j=1}^m\sqrt{p_j}\mathbb{E}S_j^3f^{(3)}(W)$ and  $\sum_{j=1}^m\sum_{k\not=j}^m\sqrt{p_j}\mathbb{E}S_jf''(W)=0$.  Hence,
\begin{align}|N_4|&=\frac{1}{2\sqrt{n}}\bigg|\sum_{j=1}^m\bigg\{\frac{3}{\sqrt{p_j}}(1-p_j)(1-2p_j)\mathbb{E}S_jf''(W)\nonumber\\
&\quad+2\bigg[\frac{1}{\sqrt{p_j}}(1-p_j)(1-2p_j)+\sqrt{p_j}\bigg]\mathbb{E}S_j^3f^{(3)}(W)\bigg\}\bigg|\nonumber\\
\label{sekcnp}&\leq\frac{1}{\sqrt{n}}\sum_{j=1}^m\frac{1}{\sqrt{p_j}}\bigg\{\frac{3}{2}|\mathbb{E}S_jf''(W)|+|\mathbb{E}S_j^3f^{(3)}(W)|\bigg\}.
\end{align}  

\medskip
\noindent{\bf{Proof Part II: Symmetry argument for optimal rate}}

\medskip 
To complete the proof, we show that the expectations $\mathbb{E}S_jf''(W)$ and $\mathbb{E}S_j^3f^{(3)}(W)$ are of order $n^{-1/2}$, and we do so by applying symmetry arguments similar to those used in the proof of Theorem \ref{chiiid}.  We use a form of multivariate normal approximation, and base our approximation on the $\mathrm{MVN}(\mathbf{0},\Sigma_{\mathbf{S}})$ Stein equation
\begin{equation}\label{testmvn2}\sum_{a=1}^m\sum_{b=1}^m\sigma_{ab}\frac{\partial^2\psi_{i}}{\partial s_a\partial s_b}(\mathbf{s})-\sum_{a=1}^ms_a\frac{\partial \psi_{i}}{\partial s_a}(\mathbf{s})=h_i(\mathbf{s})-\mathbb{E}h_i(\mathbf{Z}),
\end{equation}
where $h_1(\mathbf{s})=s_jf''(\sum_{k=1}^ms_k^2)$, $h_2(\mathbf{s})=s_j^3f^{(3)}(\sum_{k=1}^ms_k^2)$ and $\mathbf{Z}\sim \mathrm{MVN}(\mathbf{0},\Sigma_{\mathbf{S}})$.  Due to the symmetry in the test functions $h_1$ and $h_2$,  so that as $h_i(\mathbf{s})=-h_i(-\mathbf{s})$ for $i=1,2$, if $\mathbf{Z}\sim \mathrm{MVN}(\mathbf{0},\Sigma_{\mathbf{S}})$, then $\mathbb{E}h_1(\mathbf{Z})=\mathbb{E}h_2(\mathbf{Z})=0$.
Evaluating both sides of (\ref{testmvn2}) at $\mathbf{S}$ and taking expectations  gives
\begin{equation*}\mathbb{E}h_i(\mathbf{S}) =\mathbb{E}\bigg[\sum_{a=1}^m\sum_{b=1}^m\sigma_{ab}\frac{\partial^2\psi_{i}}{\partial s_a\partial s_b}(\mathbf{S})-\sum_{a=1}^mS_a\frac{\partial \psi_{i}}{\partial s_a}(\mathbf{S})\bigg].
\end{equation*}
Now Taylor expand in a similar manner to before to obtain
\[|\mathbb{E}h_i(\mathbf{S})|=|R_{5,i}+R_{6,i}|\leq |R_{5,i}|+|R_{6,i}|,\]
where
\begin{eqnarray*}R_{5,i}&=&\frac{1}{2\sqrt{n}}\sum_{a=1}^m\sum_{b=1}^m\sum_{c=1}^m\frac{1}{\sqrt{p_ap_bp_c}}\mathbb{E}\tilde{I}_a(1)I_b(1)I_c(1)\frac{\partial^3\psi_{i}}{\partial s_a\partial s_b\partial s_c}(\boldsymbol{\xi}), \\
R_{6,i}&=& -\frac{1}{\sqrt{n}}\sum_{a=1}^m\sum_{b=1}^m\sum_{c=1}^m\frac{1}{\sqrt{p_ap_bp_c}}\mathbb{E}\tilde{I}_a(1)I_b(1)\mathbb{E}I_c(1)\frac{\partial^3\psi_{i}}{\partial s_a\partial s_b\partial s_c}(\boldsymbol{\xi}).
\end{eqnarray*}
Bounding $R_{5,i}$ and $R_{6,i}$ is now almost routine:
\begin{align*}|R_{5,i}|&\leq\frac{1}{2\sqrt{n}}\bigg\{\!\sum_{a=1}^m\frac{1-p_a}{p_a^{3/2}} \mathbb{E}\bigg|I_a(1)\frac{\partial^3\psi_{i}}{\partial s_a^3}(\boldsymbol{\xi})\bigg|\!+\!\sum_{a=1}^m\sum_{b\not=a}^{m}\frac{1}{\sqrt{p_a}p_b} p_b\mathbb{E}\bigg|I_b(1)\frac{\partial^3\psi_{i}}{\partial s_a\partial s_b^2}(\boldsymbol{\xi})\bigg| \bigg\} \\
&\leq\frac{1}{2\sqrt{n}}\bigg\{\!\sum_{a=1}^m\frac{1}{p_a^{3/2}}\mathbb{E}\bigg|I_a(1)\frac{\partial^3\psi_{i}}{\partial s_a^3}(\boldsymbol{\xi})\bigg|\!+\!\sum_{a=1}^m\sum_{b\not=a}^{m}\frac{1}{\sqrt{p_a}}\mathbb{E}\bigg|I_b(1)\frac{\partial^3\psi_{i}}{\partial s_a\partial s_b^2}(\boldsymbol{\xi})\bigg| \bigg\}
\end{align*}
and
\begin{align*}|R_{6,i}|&\leq\frac{1}{\sqrt{n}}\bigg\{\sum_{a=1}^m\sum_{c=1}^m\frac{1}{\sqrt{p_c}}\mathbb{E}\bigg|I_c(1)\frac{\partial^3\psi_{i}}{\partial s_a^2\partial s_c}(\boldsymbol{\xi})\bigg| \\
&\quad+\sum_{a=1}^m\sum_{b\not=a}^{m}\sum_{c=1}^m\sqrt{\frac{p_ap_b}{p_c}}\mathbb{E}\bigg|I_c(1)\frac{\partial^3\psi_{i}}{\partial s_a\partial s_b\partial s_c}(\boldsymbol{\xi})\bigg|\bigg\}.
\end{align*}
Combining these bounds and then using that $p_j<1$ gives
\begin{align}|\mathbb{E}h_i(\mathbf{S})|&\leq\frac{1}{2\sqrt{n}}\bigg\{\sum_{a=1}^m\frac{1}{p_a^{3/2}}\mathbb{E}\bigg|I_a(1)\frac{\partial^3\psi_{i}}{\partial s_a^3}(\boldsymbol{\xi})\bigg|\!+\!3\sum_{a=1}^m\sum_{b=1}^{m}\!\frac{1}{\sqrt{p_b}}\mathbb{E}\bigg|I_a(1)\frac{\partial^3\psi_{i}}{\partial s_a^2\partial s_b}(\boldsymbol{\xi})\bigg| \nonumber \\
\label{hwssml}&\quad+2\sum_{a=1}^m\sum_{b\not=a}^{m}\sum_{c=1}^m\sqrt{\frac{p_bp_c}{p_a}}\mathbb{E}\bigg|I_a(1)\frac{\partial^3\psi_{i}}{\partial s_a\partial s_b\partial s_c}(\boldsymbol{\xi})\bigg|\bigg\}.
\end{align}

It now just remains to bound the expectations involving the derivatives $\psi_{i}$.  The bounds in the following lemma are proved in Section 5.  
\begin{lemma}\label{4psi}The third order partial derivatives of the solutions $\psi_1(\mathbf{s})$ and $\psi_2(\mathbf{s})$ of the Stein equation (\ref{testmvn2}) satisfy the following bounds
\begin{eqnarray}\bigg|\frac{\partial^3\psi_{1}}{\partial s_a\partial s_b\partial s_c}(\mathbf{s})\bigg|&\leq&\frac{\|f^{(3)}\|}{2}+\frac{4}{5}\|f^{(4)}\|\Big[8+3(s_a^2+s_b^2+s_c^2+s_j^2)\Big]\nonumber\\
\label{cmlsp1}&&+\frac{16}{35}\|f^{(5)}\|\Big[32+5(s_a^4+s_b^4+s_c^4+s_j^4)\Big], \\ 
\bigg|\frac{\partial^3\psi_{2}}{\partial s_a\partial s_b\partial s_c}(\mathbf{s})\bigg|&\leq& \frac{\|f^{(3)}\|}{2}+\frac{12}{5}\|f^{(4)}\|\Big[4+s_a^2+s_b^2+s_c^2+3s_j^2\Big]\nonumber\\
&&+\frac{8}{35}\|f^{(5)}\|\Big[384+5(7s_a^4+7s_b^4+7s_c^4+27s_j^4)\Big]\nonumber\\
&&+
\label{cmlsp2}\|f^{(6)}\|\bigg[\frac{4096}{21}+\frac{128}{27}(s_a^6+s_b^6+s_c^6+3s_j^6)\bigg].
\end{eqnarray}
\end{lemma}

On applying Lemmas \ref{qqqppp} and \ref{4psi}, we now have
\begin{align*}&\mathbb{E}\bigg|I_a(1)\frac{\partial^3\psi_{1}}{\partial s_a\partial s_b\partial s_c}(\boldsymbol{\xi})\bigg|\\
&\leq \frac{\|f^{(3)}\|}{2}\mathbb{E}I_a(1)+\frac{4}{5}\|f^{(4)}\|\Big[8\mathbb{E}I_a(1)+3\Big(\mathbb{E}I_a\xi_a^2+\mathbb{E}I_a\xi_b^2+\mathbb{E}I_a\xi_c^2+\mathbb{E}I_a\xi_j^2\Big)\Big]  \\
&\quad+\frac{16}{35}\|f^{(5)}\|\Big[32\mathbb{E}I_a(1)+5\Big(\mathbb{E}I_a\xi_a^4+\mathbb{E}I_a\xi_b^4+\mathbb{E}I_a\xi_c^4+\mathbb{E}I_a\xi_j^4\Big)\Big]  \\
&\leq p_a\bigg[\frac{\|f^{(3)}\|}{2}+\frac{4}{5}(8+3\cdot4\cdot4)\|f^{(4)}\|+\frac{16}{35}(32+5\cdot4\cdot27)\|f^{(5)}\|\bigg] \\
&= p_a\bigg[\frac{\|f^{(3)}\|}{2}+\frac{224}{5}\|f^{(4)}\|+\frac{9152}{35}\|f^{(5)}\|\bigg]
\end{align*}
and
\begin{align*}&\mathbb{E}\bigg|I_a(1)\frac{\partial^3\psi_{2}}{\partial s_a\partial s_b\partial s_c}(\boldsymbol{\xi})\bigg|\\
&\leq p_a\bigg[\frac{\|f^{(3)}\|}{2}+\frac{12}{5}(4+6\cdot3)\|f^{(4)}\|+\frac{8}{35}(384+5\cdot6\cdot27)\|f^{(5)}\|\\
&\quad+\|f^{(6)}\|\bigg(\frac{4096}{21}+\frac{128}{7}\cdot6\cdot305\bigg)\bigg]\\
&= p_a\bigg[\frac{\|f^{(3)}\|}{2}+\frac{336}{5}\|f^{(4)}\|+\frac{9552}{35}\|f^{(5)}\|+\frac{706816}{21}\|f^{(6)}\|\bigg].
\end{align*}
Substituting into (\ref{hwssml}), 
\begin{align*}|\mathbb{E}h_1(\mathbf{S})|&\leq\frac{1}{2\sqrt{n}}\bigg[\frac{\|f^{(3)}\|}{2}+\frac{224}{5}\|f^{(4)}\|+\frac{9152}{35}\|f^{(5)}\|\bigg]\bigg\{\sum_{a=1}^m\frac{1}{\sqrt{p_a}}\\
&\quad+3\sum_{a=1}^m\sum_{b=1}^{m}\frac{p_a}{\sqrt{p_b}}+2\sum_{a=1}^m\sum_{b\not=a}^{m}\sum_{c=1}^m\sqrt{p_ap_bp_c}\bigg\}\\
&\leq\frac{1}{\sqrt{n}}\{2\|f^{(3)}\|+135\|f^{(4)}\|+785\|f^{(5)}\|\}\sum_{a=1}^m\frac{1}{\sqrt{p_a}}
\end{align*}
and, by a similar calculation,
\begin{equation*}|\mathbb{E}h_2(\mathbf{S})|\leq\frac{1}{\sqrt{n}}\{2\|f^{(3)}\|+202\|f^{(4)}\|+819\|f^{(5)}\|+100974\|f^{(6)}\|\}\sum_{a=1}^m\frac{1}{\sqrt{p_a}},
\end{equation*}
where we rounded the constants up to the nearest integer.  Finally, we substitute these inequalities into (\ref{sekcnp}) to bound $N_4$:
\begin{align*}|N_4|&\leq\frac{1}{n}\bigg\{\bigg(\frac{3}{2}\cdot2+2\bigg)\|f^{(3)}\|+\bigg(\frac{3}{2}\cdot135+202\bigg)\|f^{(4)}\|\\
&\quad+\bigg(\frac{3}{2}\cdot785+819\bigg)\|f^{(5)}\|+100974\|f^{(6)}\|\bigg\}\bigg(\sum_{j=1}^m\frac{1}{\sqrt{p_j}}\bigg)^2 \\
&\leq\frac{1}{n}\{5\|f^{(3)}\|+405\|f^{(4)}\|+1997\|f^{(5)}\|+100974\|f^{(6)}\|\}\bigg(\sum_{j=1}^m\frac{1}{\sqrt{p_j}}\bigg)^2.
\end{align*}

In conclusion, 
\[|\mathbb{E}h(W)-\chi_{(m-1)}^2h|\leq |R_1|+|R_2|+|R_3|+|R_4|+|N_4|.\]
To arrive at (\ref{pearbound}), we sum up these remainders and use the inequality $\sum_{j=1}^m\frac{1}{p_j}\leq(\sum_{j=1}^mp_j^{-1/2})^2$ to simplify the bound.  Finally, we use (\ref{chibound}) to translate bounds on the derivatives of the solution $f$ to bounds on the derivatives of the test function $h$, which completes the proof of Theorem \ref{Pearson_thm_non_int}.  \hfill $\square$

\section{Further proofs} \label{appendix}

\noindent\emph{Proof of Lemma \ref{normaltech}.} From Lemma I.4 in  \cite{stein2}, $\psi$ exists and 
\[
\psi(w) = \mathrm{e}^{\frac12 w^2} \int_w^\infty g^{(3)}(s) \mathrm{e}^{-\frac12 s^2} \, \mathrm{d}s = - \mathrm{e}^{\frac12 w^2} \int_{-\infty}^w  g^{(3)}(s) \mathrm{e}^{-\frac12 s^2} \, \mathrm{d}s.
\]
From  $g^{(3)}(s) = 3s f''(s^2) + 2 s^3 f^{(3)}(s^2)$   we get that 
\bea \label{gbound2}
|  g^{(3)}(x) |  \le  3 |x| \| f'' \| + 2 |x|^3 \| f^{(3)}\|. 
\ena
Hence for $ w>0$ 
\begin{align*}
| \psi(w) | &\leq  3 \|f''\|  \mathrm{e}^{\frac12 w^2} \int_w^\infty  s \mathrm{e}^{-\frac12 s^2} \, \mathrm{d}s  + 2 \| f^{(3)}\| \mathrm{e}^{\frac12 w^2} \int_w^\infty  s^3 \mathrm{e}^{-\frac12 s^2} \, \mathrm{d}s  \\
&=3 \|f'' \| +  2 (w^2+2)  \| f^{(3)}\|   .
\end{align*}
Similarly, 
for $w<0$ 
\begin{align*}
| \psi(w) | &\le 3 \|f''\|  \mathrm{e}^{\frac12 w^2} \int_{-\infty}^w  (-s) \mathrm{e}^{-\frac12 s^2} \, \mathrm{d}s  + 2 \| f^{(3)}\| \mathrm{e}^{\frac12 w^2} \int_{-\infty}^w (- s)^3 \mathrm{e}^{-\frac12 s^2} \, \mathrm{d}s  \\
&=  3 \|f'' \| +  2 (w^2+2)  \| f^{(3)}\|  ,
\end{align*} 
proving the first bound. The second bound uses \eqref{Stein-normal} to obtain 
$$ x \psi'(x) = x^2 \psi(x) + x g^{(3)}(x) $$
and the bound then follows directly with \eqref{bound1} and \eqref{gbound2}. For the last bound differentiate \eqref{Stein-normal} to get
$$ \psi''(x) = x \psi ' (x) + \psi (x) + g^{(4)}(x)$$
and combine \eqref{bound1}, \eqref{bound2} and \eqref{gbound}.  \hfill $\square$

\vspace{3mm}

\noindent\emph{Proof of Lemma \ref{operatorcomp}.} The derivatives of $g$ are
\begin{eqnarray*}
\frac{\partial g}{\partial s_j} (\mathbf{s}) &=&  \frac{1}{2}s_jf'(w), \\
\frac{\partial^2g}{\partial s_j^2} (\mathbf{s}) &=&  \frac{1}{2}f'(w)+s_j^2f''(w), \\
\frac{\partial^2g}{\partial s_j \partial s_k} (\mathbf{s}) &=& s_js_kf''(w), \qquad  j \neq k. 
\end{eqnarray*}
Note that we can write the multivariate normal Stein equation (\ref{mvn1}) as 
$$
\mathcal{A}_{\mathrm{MVN}(\mathbf{0},\Sigma_\mathbf{S})}   g(\mathbf{s}) = \sum_{j=1}^m \sigma_{jj} \frac{\partial^2g}{\partial s_j^2} (\mathbf{s}) + \sum_{j=1}^m \sum_{k \neq j}^m \sigma_{jk} \frac{\partial^2g}{\partial s_j \partial s_k} (\mathbf{s}) - \sum_{j=1}^m s_j \frac{\partial g}{\partial s_j} (\mathbf{s}). $$
Now, as $\sum_{j=1}^mp_j=1$,
\begin{align*}
\sum_{j=1}^m\sigma_{jj} \frac{\partial^2g}{\partial s_j^2} (\mathbf{s}) &= \frac{1}{2}f'(w)\sum_{j=1}^m(1-p_j)+f''(w)\sum_{j=1}^m(1-p_j)s_j^2\\
&=\frac{m-1}{2}f'(w)+wf''(w)-f''(w)\sum_{j=1}^mp_js_j^2. \end{align*}
Similarly, since $\sum_{j=1}^m \sum_{k \neq j}^m \sqrt{p_jp_k}s_j s_k = -\sum_{j=1}^mp_j s_j^2$ (recall $\sum_{j=1}^m \sqrt{p_j}s_j = 0$),
\begin{equation*}
\sum_{j=1}^m \sum_{k \neq j}^m\sigma_{jk} \frac{\partial^2g}{\partial s_j \partial y_k} (\mathbf{s}) = -f''(w)\sum_{j=1}^m \sum_{k \neq j}^m\sqrt{p_jp_k}s_js_k=f''(w)\sum_{j=1}^mp_js_j^2. \end{equation*}
Lastly,
\begin{equation*}
\sum_{j=1}^m s_j \frac{\partial g}{\partial s_j} (\mathbf{s}) = \frac{1}{2}f'(w)\sum_{j=1}^ms_j^2=\frac{1}{2}wf'(w).
\end{equation*} 
Putting all the above together gives that 
\[\mathcal{A}_{\mathrm{MVN}(\mathbf{0},\Sigma_{\mathbf{S}})}g(\mathbf{s})=wf''(w)+\frac{1}{2}(m-1-w)f'(w)=  \mathcal{A}_{m-1} f(w),
\] 
as required. \hfill $\square$

\vspace{3mm}

\noindent\emph{Proof of Lemma \ref{qqqppp}.} Suppose firstly that $j=k$.  As $S_j^{(1)}$ and $I_j(1)$ are independent, we have
\begin{align*}\mathbb{E}I_j(1)\xi_j^2&=\mathbb{E}I_j(1)\bigg(S_j^{(1)}+\frac{\theta_{j}}{\sqrt{np_j}}I_j(1)\bigg)^2 \\
&\leq \mathbb{E}I_j(1)\Big\{\mathbb{E}(S_j^{(1)})^2+2\mathbb{E}|S_j^{(1)}|+1\Big\}<4p_j,
\end{align*}
where we used that $0<\theta_j<1$ and $np_j\geq 1$ to obtain the first inequality.  Similarly,
\begin{align*}\mathbb{E}I_j(1)\xi_j^4&\leq p_j\Big[\mathbb{E}(S_j^{(1)})^4+4\mathbb{E}|S_j^{(1)}|^3+6\mathbb{E}(S_j^{(1)})^2+4\mathbb{E}|S_j^{(1)}|+1\Big]\\
&<p_j(4+4\cdot 4^{3/4}+6+4+1)<27p_j,
\end{align*}
and
\begin{align*}\mathbb{E}I_j(1)\xi_j^6&\leq p_j\Big[\mathbb{E}(S_j^{(1)})^6+6\mathbb{E}|S_j^{(1)}|^5+15\mathbb{E}(S_j^{(1)})^4+20\mathbb{E}|S_j^{(1)}|^3\\
&\quad+15\mathbb{E}(S_j^{(1)})^2+6\mathbb{E}|S_j^{(1)}|+1\Big]\\
&<p_j(42+6\cdot42^{5/6}+15\cdot4+20\cdot 4^{3/4}+15+6+1)<305p_j,
\end{align*}
where we used Lemma \ref{thismoment} to bound the absolute moments of $S_j^{(1)}$.  Also,
\begin{equation*}\mathbb{E}|I_j(1)\xi_j|\leq p_j(\mathbb{E}|S_j^{(1)}|+1)<2p_j,
\end{equation*}
and 
\begin{equation*}\mathbb{E}|I_j(1)\xi_j^3|\leq p_j\Big[\mathbb{E}|S_j^{(1)}|^3+3\mathbb{E}(S_j^{(1)})^2+3\mathbb{E}|S_j^{(1)}|+1\Big]\!<p_j(4^{3/4}+3+3+1)\!<14p_j.
\end{equation*}
When $j\not=k$, it is clear that the same bounds still hold, as $I_j(1)I_k(1)=0$. \hfill $\square$

\vspace{3mm}

\noindent\emph{Proof of Lemma \ref{4psi}.} Since the covariance-matrix $\Sigma_{\mathbf{S}}$ is non-negative definite, the solution of the Stein equation (\ref{testmvn2}) is well-defined and is given by (see \cite{meckes}):
\begin{equation*}\psi_i(\mathbf{s})=-\int_0^{\infty}\mathbb{E}[h_i(\mathrm{e}^{-u}\mathbf{s}+\sqrt{1-\mathrm{e}^{-2u}}\mathbf{Z})]\,\mathrm{d}u.
\end{equation*}
By dominated convergence, 
\begin{equation*}\frac{\partial^3\psi_i}{\partial s_a\partial s_b\partial s_c}(\mathbf{s})=-\int_0^{\infty}\mathrm{e}^{-3u}\mathbb{E}\bigg[\frac{\partial^3h_i}{\partial s_a\partial s_b\partial s_c}(\mathrm{e}^{-u}\mathbf{s}+\sqrt{1-\mathrm{e}^{-2u}}\mathbf{Z})\bigg]\,\mathrm{d}u,
\end{equation*}
and so
\begin{equation*}\bigg|\frac{\partial^3\psi_i}{\partial s_a\partial s_b\partial s_c}(\mathbf{s})\bigg|\leq\int_0^{\infty}\mathrm{e}^{-3u}\mathbb{E}\bigg|\frac{\partial^3h_i}{\partial s_a\partial s_b\partial s_c}(\mathrm{e}^{-u}\mathbf{s}+\sqrt{1-\mathrm{e}^{-2u}}\mathbf{Z})\bigg|\,\mathrm{d}u.
\end{equation*}

We now obtain bounds for the third order partial derivatives of $h_1$ and $h_2$.  By straightforward differentiation,
\begin{align*}\frac{\partial^3h_1}{\partial s_a\partial s_b\partial s_c}(\mathbf{s})&=2[\delta_{ja}+\delta_{jb}+\delta_{jc}]f^{(3)}(w)+4[s_j(s_a+s_b+s_c)\\
&\quad+s_bs_c\delta_{ja}+s_as_c\delta_{jb}+s_as_b\delta_{jc}]f^{(4)}(w)+8s_js_as_bs_cf^{(5)}(w), \\
\frac{\partial^3h_2}{\partial s_a\partial s_b\partial s_c}(\mathbf{s})&=6\delta_{ja}\delta_{jb}\delta_{jc}f^{(3)}(w)\!+\!12s_j[s_a\delta_{jb}\delta_{jc}+s_b\delta_{ja}\delta_{jc}+s_c\delta_{ja}\delta_{jb}]f^{(4)}(w)\\
&\quad+\!4[s_j^3(s_a+s_b+s_c)\!+\!3s_j^2(s_bs_c\delta_{ja}+s_as_c\delta_{jb}+s_as_b\delta_{jc})]f^{(5)}(w)\\
&\quad+\!8s_j^3s_as_bs_cf^{(6)}(w),
\end{align*}
where $\delta_{jj}=1$ and $\delta_{jk}=0$ if $j\not=k$.  We can bound these partial derivatives by using the inequalities $\delta_{jk}\leq 1$ and $\prod_{k=1}^n|a_k|\leq \frac{1}{n}\sum_{k=1}^n|a_k|^n$ for $n\geq1$.  Doing so yields the bounds
\begin{align*}\bigg|\frac{\partial^3h_1}{\partial s_a\partial s_b\partial s_c}(\mathbf{s})\bigg|&\leq6\|f^{(3)}\|+6\|f^{(4)}\|(s_a^2+s_b^2+s_c^2+s_j^2)\\
&\quad+2\|f^{(5)}\|(s_a^4+s_b^4+s_c^4+s_j^4), \\
\bigg|\frac{\partial^3h_2}{\partial s_a\partial s_b\partial s_c}(\mathbf{s})\bigg|&\leq 6\|f^{(3)}\|+6\|f^{(4)}\|(s_a^2+s_b^2+s_c^2+3s_j^2)\\
&\quad+\|f^{(5)}\|(7s_s^4+7s_b^4+7s_c^4+27s_j^4)\\
&\quad+\frac{4}{3}\|f^{(6)}\|(s_a^6+s_b^6+s_c^6+3s_j^6).
\end{align*}

We now use the inequality for the third order partial derivative of $h_2(\mathbf{s})$ to bound the third order partial derivatives of $\psi_2(\mathbf{s})$.  The random vector $\mathbf{Z}\sim\mathrm{MVN}(\mathbf{0},\Sigma_{\mathbf{S}})$ can be written as $(Z_1,\ldots,Z_m)$, where $Z_j\sim N(0,1-p_j)$ and $\mathrm{Cov}(Z_j,Z_k)=-\sqrt{p_jp_k}$ for $j\not=k$.  On applying the inequality $|a_1+a_2|^r\leq 2^{r-1}(|a_1|^r+|a_2|^r)$, where $r\geq1$, we have
\begin{align*}&\bigg|\frac{\partial^3\psi_i}{\partial s_a\partial s_b\partial s_c}(\mathbf{s})\bigg|\\
&\leq\int_0^{\infty}\mathrm{e}^{-3u}\mathbb{E}\bigg[6\|f^{(3)}\|+12\|f^{(4)}\|\Big[\mathrm{e}^{-2u}(s_a^2+s_b^2+s_c^2+3s_j^2)\\
&\quad+(1-\mathrm{e}^{-2u})(Z_a^2+Z_b^2+Z_c^2+3Z_j^2)\Big]
+8\|f^{(5)}\|\Big[\mathrm{e}^{-4u}(7s_a^4+7s_b^4+7s_c^4\\
&\quad+27s_j^4)+(1-\mathrm{e}^{-2u})^2(7Z_a^4+7Z_b^4+7Z_c^4+27Z_j^4)\Big]
+\frac{128}{3}\|f^{(6)}\|\Big[\mathrm{e}^{-6u}(s_a^6\\
&\quad+s_b^6+s_c^6+3s_j^6)+(1-\mathrm{e}^{-2u})^3(Z_a^6+Z_b^6+Z_c^6+3Z_j^6)\Big]\bigg]\,\mathrm{d}u\\
&\leq\frac{\|f^{(3)}\|}{2}+\frac{12}{5}\|f^{(4)}\|\Big[4+s_a^2+s_b^2+s_c^2+3s_j^2\Big]+\frac{8}{35}\|f^{(5)}\|\Big[384+5(7s_a^4+7s_b^4\\
&\quad+7s_c^4+27s_j^4)\Big]+
\|f^{(6)}\|\bigg[\frac{4096}{21}+\frac{128}{27}(s_a^6+s_b^6+s_c^6+3s_j^6)\bigg].
\end{align*}  
To obtain the last inequality, we used that $\mathbb{E}Z_j^2=1-p_j<1$, $\mathbb{E}Z_j^4=3(1-p_j)^2<3$ and $\mathbb{E}Z_j^6=15(1-p_j)^3<15$, $j=1,\ldots,m$, and the formulas $\int_0^\infty \mathrm{e}^{-3u}(1-\mathrm{e}^{-2u})\,\mathrm{du}=\frac{2}{15}$, $\int_0^\infty \mathrm{e}^{-3u}(1-\mathrm{e}^{-2u})^2\,\mathrm{du}=\frac{8}{105}$ and $\int_0^\infty \mathrm{e}^{-3u}(1-\mathrm{e}^{-2u})^3\,\mathrm{du}=\frac{16}{315}$.  This completes the proof of inequality (\ref{cmlsp2}), and inequality (\ref{cmlsp1}) follows from a similar calculation. \hfill $\square$

\vspace{3mm}

\noindent\emph{Proof of Theorem \ref{pearsqrtn}.} As was the case in the proof of Theorem \ref{Pearson_thm_non_int}, we require a bound for $\mathbb{E}\mathcal{A}_{m-1}f(W)$, which is equivalent to bounding $\mathbb{E}\mathcal{A}_{\mathrm{MVN}(\mathbf{0},\Sigma_\mathbf{S})}g(\mathbf{S})$.  By using Taylor expansions in a similar manner to that used in the proof of Theorem \ref{Pearson_thm_non_int}, we have that
\begin{equation*}|\mathbb{E}h(W)-\chi_{(m-1)}^2h|=|\mathbb{E}\mathcal{A}_{m-1}f(W)|=|\mathbb{E}\mathcal{A}_{\mathrm{MVN}(\mathbf{0},\Sigma_\mathbf{S})}g(\mathbf{S})|\leq |R_1|+|R_2|,
\end{equation*}
where
\begin{align*}|R_1|&\leq \frac{1}{2\sqrt{n}}\bigg|\sum_{j=1}^m\sum_{k=1}^m\sum_{l=1}^{m}\frac{1}{\sqrt{p_jp_kp_l}}\mathbb{E}\tilde{I}_j(1)I_k(1)I_l(1)\frac{\partial^3g}{\partial s_j\partial s_k\partial s_l}(\boldsymbol{\xi})\bigg|, \\
|R_2|&\leq \frac{1}{\sqrt{n}}\bigg|\sum_{j=1}^m\sum_{k=1}^m\sum_{l=1}^{m}\frac{1}{\sqrt{p_jp_kp_l}}\mathbb{E}\tilde{I}_j(1)I_k(1)\mathbb{E}I_l(1)\frac{\partial^3g}{\partial s_j\partial s_k\partial s_l}(\boldsymbol{\xi})\bigg|.
\end{align*}
and again $\boldsymbol{\xi}$ denotes a vector with $m$ entries with $j$-th entry $\xi_j=S_j^{(1)}+\frac{\theta_j}{\sqrt{np_j}}I_j(1)$ for some $\theta_j\in(0,1)$.  Carrying out a calculation similar to the one used to obtain (\ref{re123z}) yields 
\begin{align}|R_1|\!+\!|R_2|&\leq\frac{1}{2\sqrt{n}}\bigg\{\!\sum_{j=1}^m\frac{1}{p_j^{3/2}}\mathbb{E}\bigg|I_j(1)\frac{\partial^3g}{\partial s_j^3}(\boldsymbol{\xi})\bigg| \!+\!2\sum_{j=1}^m\sum_{k=1}^m\!\frac{1}{\sqrt{p_j}}\mathbb{E}\bigg|I_j(1)\frac{\partial^3g}{\partial s_j\partial s_k^2}(\boldsymbol{\xi})\bigg|\nonumber \\
&\quad+\sum_{j=1}^m\sum_{k\not=j}^m\frac{\sqrt{p_k}}{p_j}\mathbb{E}\bigg|I_j(1)\frac{\partial^3g}{\partial s_j^2\partial s_k}(\boldsymbol{\xi})\bigg| \nonumber\\
\label{rhsexpect}&\quad+2\sum_{j=1}^m\sum_{k\not=j}^m\sum_{l=1}^m\sqrt{\frac{p_kp_l}{p_j}}\mathbb{E}\bigg|I_j(1)\frac{\partial^3g}{\partial s_j\partial s_k\partial s_l}(\boldsymbol{\xi})\bigg|\bigg\}.
\end{align}

To complete the proof, we require bounds on the expectations on the right-hand side of (\ref{rhsexpect}).  Now recall that $g(\mathbf{s})=\frac{1}{4}f(w)$.  By a straightforward differentiation,
\[\frac{\partial^3g}{\partial s_j^3}(\mathbf{s})=3s_jf''(w)+2s_j^2f^{(3)}(w),\]
and so, for any $j,k,l$, 
\begin{equation}\label{partderboundz}\bigg|\frac{\partial^3g}{\partial s_j\partial s_k\partial s_l}(\mathbf{s})\bigg|\leq (|s_j|+|s_k|+|s_l|)\|f''\|+\frac{2}{3}(|s_j|^3+|s_k|^3+|s_l|^3)\|f^{(3)}\|.
\end{equation}
To bound the expectations, we also use the inequalities $\mathbb{E}|I_j(1)\xi_k|<2p_j$ and $\mathbb{E}|I_j(1)\xi_k^3|<14p_j$, $j,k=1,\ldots,m$ from Lemma \ref{qqqppp}.  Using (\ref{partderboundz}) and these inequalities yields
\begin{equation*}\mathbb{E}\bigg|I_j(1)\frac{\partial^3g}{\partial s_j\partial s_k\partial s_l}(\boldsymbol{\xi})\bigg|\leq 3\cdot 2p_j\|f''\|+2\cdot 14p_j\|f^{(3)}\|=2p_j\{3\|f''\|+14\|f^{(3)}\|\}.
\end{equation*}
Hence, we obtain the bound
\begin{align}|\mathbb{E}h(W)-\chi_{(m-1)}^2h|&\leq\frac{1}{\sqrt{n}}\{3\|f''\|+14\|f^{(3)}\|\}\bigg\{\sum_{j=1}^m\frac{1}{\sqrt{p_j}}+2m\sum_{j=1}^m\sqrt{p_j}\nonumber\\
&\quad+\sum_{j=1}^m\sum_{k\not=j}^m\sqrt{p_k}+2\sum_{j=1}^m\sum_{k\not=j}^m\sum_{l=1}^m\sqrt{p_jp_kp_l}\bigg\} \nonumber\\
\label{nrfgt}&\leq\frac{6}{\sqrt{n}}\{3\|f''\|+14\|f^{(3)}\|\}\sum_{j=1}^m\frac{1}{\sqrt{p_j}}.
\end{align}
Using inequality (\ref{chibound}) to translate bounds for the derivatives of the solution $f$ to bounds on the derivatives of the test function $h$ completes the proof.    \hfill $\square$

\vspace{3mm}

\noindent\emph{Proof of Corollary \ref{kolcol}.} Let $\alpha>0$, and for some fixed $z>0$ define
\begin{equation*}h_\alpha(x)=\begin{cases} 1, & \: \mbox{if } x\leq z, \\
1-2(x-z)^2/\alpha^2, & \:  \mbox {if } z<x\leq z+\alpha/2, \\
2(x-(z+\alpha))^2/\alpha^2, & \:  \mbox {if } z+\alpha/2<x\leq z+\alpha, \\
0, & \:  \mbox {if } x\geq z+\alpha. \end{cases}
\end{equation*}
Then $h_\alpha'$ exists and is Lipshitz continuous with $\|h_\alpha\|=1$, $\|h_\alpha'\|=2/\alpha$ and $\|h_\alpha''\|=4/\alpha^2$.  Let $Y_d$ be a $\chi_{(d)}^2$ random variable, then, by (\ref{pearbound01}),
\begin{align}&\mathbb{P}(W\leq z)-\mathbb{P}(Y_{m-1}\leq z)\nonumber\\
&\leq \mathbb{E}h_\alpha(W)-\mathbb{E}h_\alpha(Y_{m-1})+\mathbb{E}h_\alpha(Y_{m-1})-\mathbb{P}(Y_{m-1}\leq z)\nonumber \\
&\leq \frac{12}{\sqrt{np_*}}\{6\|h_\alpha\|+46\|h_\alpha'\|+84\|h_\alpha''\|\}+\mathbb{P}(z\leq Y_{m-1}\leq z+\alpha)\nonumber \\
\label{kol123}&=\frac{12}{\sqrt{np_*}}\bigg\{6+\frac{92}{\alpha}+\frac{336}{\alpha^2}\bigg\}+\mathbb{P}(z\leq Y_d\leq z+\alpha).
\end{align}

Now, for $d=1$ (which corresponds to $m=2$), 
\begin{equation*}\mathbb{P}(z\leq Y_1\leq z+\alpha)=\int_z^{z+\alpha}\frac{\mathrm{e}^{-x/2}}{\sqrt{2\pi x}}\,\mathrm{d}x \leq \int_0^{\alpha}\frac{1}{\sqrt{2\pi x}}\,\mathrm{d}x=\sqrt{\frac{2\alpha}{\pi}}.
\end{equation*}
For $d\geq2$, the mode of $Y_d$ is given by $d-2$.  The density of $Y_2$ is clearly bounded by $\frac{1}{2}$, and, for $d\geq3$, the density of $Y_d$ can be bounded by
\begin{align*}\frac{1}{2^{d/2}\Gamma(\frac{d}{2})}x^{d/2-1}\mathrm{e}^{-x/2}\leq\frac{1}{2^{d/2}\Gamma(\frac{d}{2})}(d-2)^{d/2-1}\mathrm{e}^{-(d-2)/2}\leq\frac{1}{2\sqrt{\pi(d-2)}},
\end{align*}
where the last inequality follows from Stirling's inequality $\Gamma(x+1)\geq \sqrt{2\pi}x^{x+1/2}\mathrm{e}^{-x}$, which holds for all $x>0$.  Therefore
\begin{equation}\label{ccbbcsb}\mathbb{P}(z\leq Y_{m-1}\leq z+\alpha)\leq \begin{cases} \displaystyle \sqrt{2\alpha/\pi}, & \: \mbox{if } m=2, \\
\alpha/2, & \: \mbox{if } m=3, \\
\displaystyle \frac{\alpha}{2\sqrt{\pi(m-3)}}, & \:  \mbox {if } m\geq4. \end{cases}
\end{equation}

Bounds for $m=2$, $m=3$ and $m\geq 4$ now follow on substituting inequality (\ref{ccbbcsb}) into (\ref{kol123}) and choosing an appropriate $\alpha$.  For $m=2$, we take $\alpha=52.75n^{-1/5}$; for $m=3$, we choose $\alpha=25.27n^{-1/6}$; and $\alpha=30.58(m-3)^{1/6}n^{-1/6}$ is taken when $m\geq4$.  We can obtain a lower bound similarly, which is the negative of the upper bound.  The proof is now complete. \hfill $\square$

\section*{Acknowledgements}
The problem was first brought to our attention by Persi Diaconis, and we are thankful for that. As so many, this  paper would not exist without him.  Larry Goldstein provided many fruitful discussions. During the course of this research, RG was supported by an EPSRC DPhil Studentship, an EPSRC Doctoral Prize, and  EPSRC grant EP/K032402/1. AP was funded in part by EPSRC grant GR/R52183/01. GR acknowledges support from  EPSRC grants GR/R52183/01 and  EP/K032402/1. We would also like to thank an anonymous referee for their report.


\begin{thebibliography}{99}
\bibitem{pearson} Pearson, K.  On the criterion that a given system of deviations is such that it can be reasonably supposed to have arisen from random sampling.  \emph{Phil. Mag.} $\mathbf{50}$ (1900), pp. 157--175.
\end{thebibliography}
\end{document}